\documentclass[11pt,reqno]{amsart}
\usepackage{enumerate}
\usepackage{amsfonts, amsmath, amssymb, amscd, amsthm, bm}

\usepackage{enumerate}
\usepackage{url}
\usepackage[linktocpage=true,colorlinks,citecolor=magenta,linkcolor=blue,urlcolor=magenta]{hyperref}
\usepackage{multicol}
\usepackage{comment}

 \newtheorem{thm}{Theorem}[section]

 \newtheorem{lem}[thm]{Lemma}
 \newtheorem{assump}[thm]{Assumption}
 \newtheorem{prop}[thm]{Proposition}
 \theoremstyle{definition}

 \newtheorem{rem}[thm]{Remark}
 \newtheorem*{ack}{Acknowledgments}

 \theoremstyle{claim}
 \numberwithin{equation}{section}
\numberwithin{equation}{section}

\newcounter{rom}
\renewcommand{\therom}{(\roman{rom})}
{\end{list}}

\title{Surfaces expanding by non-concave curvature functions}
\begin{document}

\author[H. Li]{Haizhong Li}
\address{Department of Mathematical Sciences,
Tsinghua University,  Beijing 100084, P. R. China}
\email{\href{mailto:hli@math.tsinghua.edu.cn}{hli@math.tsinghua.edu.cn}}

\author[X. Wang]{Xianfeng Wang}
\address{School of Mathematical Sciences and LPMC,
Nankai University,
Tianjin 300071,  P. R. China}
\email{\href{mailto:wangxianfeng@nankai.edu.cn}{wangxianfeng@nankai.edu.cn}}

\author[Y. Wei]{Yong Wei}
\address{Mathematical Sciences Institute,
Australian National University, Canberra,
ACT 2601 Australia}
\email{\href{mailto:yong.wei@anu.edu.au}{yong.wei@anu.edu.au}}

\keywords {Surface, space form, inverse curvature flow, non-concave curvature function.}
\subjclass[2010]{Primary 53C44; Secondary 53C21, 58J35}


\begin{abstract}
In this paper, we first investigate the flow of convex surfaces in the space form $\mathbb{R}^3(\kappa)~(\kappa=0,1,-1)$  expanding by $F^{-\alpha}$, where $F$ is a smooth, symmetric, increasing and homogeneous of degree one function of the principal curvatures of the surfaces and the power $\alpha\in(0,1]$ for $\kappa=0,-1$ and $\alpha=1$ for $\kappa=1$. By deriving that the pinching ratio of the flow surface $M_t$ is no greater than that of the initial surface $M_0$, we prove the long time existence and the convergence of the flow. No concavity assumption of $F$ is required. We also show that for the flow in $\mathbb{H}^3$ with $\alpha\in (0,1)$, the limit shape may not be necessarily round after rescaling.
\end{abstract}

\maketitle

\section{Introduction}
Let $\mathbb{R}^3(\kappa)~(\kappa=0,1,-1)$ be a real space form, i.e., when $\kappa=0$, $\mathbb{R}^3(0)=\mathbb{R}^3$, when $\kappa=1$, $\mathbb{R}^3(1)=\mathbb{S}^3$, and when $\kappa=-1$, $\mathbb{R}^3(-1)=\mathbb{H}^3$.   Given a compact smooth immersion $X_0:M\rightarrow \mathbb{R}^3(\kappa)$, we consider the smooth family of immersions $X(x,t):M\times [0,T)\to\mathbb{R}^3(\kappa)$ solving the evolution equation
\begin{equation}\label{1.1}
 \left\{\begin{aligned}
 \frac{\partial}{\partial t}X(x,t)=&~F^{-\alpha}(x,t)\nu(x,t),\\
 X(\cdot,0)=&~X_0(\cdot),
  \end{aligned}\right.
 \end{equation}
 where $F(x,t)=F(\lambda_1(x,t),\lambda_2(x,t))$ is a smooth symmetric function of the principal curvatures of the surfaces and $\nu$ is the outer unit normal of $M_t=X_t(M)$. Throughout this paper, we assume that $F$ satisfies the following conditions:

\begin{assump}\label{assum-1}
Let $\Gamma_+=\{(\lambda_1,\lambda_2)\in \mathbb{R}^2: \lambda_i>0, i=1,2\}$ be the
positive quadrant in $\mathbb{R}^2$. Assume that
\begin{itemize}
  \item[(i)] $F$ is smooth, symmetric and positive on $\Gamma_+$.
  \item[(ii)] $F$ is strictly increasing in each argument, i.e., ${\partial F}/{\partial \lambda_i}>0$ on $\Gamma_+$, $\forall~i=1,2$.
  \item[(iii)] $F$ is homogeneous of degree $1$, i.e., $F(k\lambda)=kF(\lambda)$ for any $k>0$ and $\lambda=(\lambda_1,\lambda_2)\in \Gamma_+$.
  \item[(iv)] $F$ is normalized such that $F(1,1)=2$.
\end{itemize}
\end{assump}
We refer the reader to \cite{Andrews2015convexity} for examples of $F$ satisfying Assumption \ref{assum-1}.

For strictly convex surfaces in $\mathbb{R}^3(\kappa)$, \eqref{1.1} is a parabolic equation and  has a smooth solution on a maximal time interval $[0,T)~(T\leq\infty)$  for any $F$ satisfying Assumption \ref{assum-1} (cf. \cite{HA1999}). In this paper, we will first study the long time behavior of flow \eqref{1.1}.

\begin{thm}\label{thm1.1}
For any smooth, closed strictly convex surface $M_0$ in $\mathbb{R}^3$,  there  exists a smooth
solution of flow \eqref{1.1} with $F$ satisfying  Assumption \ref{assum-1} and $\alpha\in(0,1]$. The solution exists for all time $t\in[0,\infty)$, $M_t$ converges to infinity as $t\to \infty$ and the properly rescaled surfaces converge exponentially in $C^{\infty}$-topology to the unit sphere $\mathbb{S}^2$.
\end{thm}

\begin{thm}\label{thm1.2}
For any smooth, closed strictly convex surface $M_0$ in $\mathbb{H}^3$, there  exists a smooth
solution of flow \eqref{1.1} with $F$ satisfying  Assumption \ref{assum-1} and $\alpha\in(0,1]$. The solution exists for all time $t\in[0,\infty)$, and each surface $M_t$ can be written as a graph of a function $u(t,\theta)$ over $\mathbb{S}^2$. The principal curvatures $\lambda_i(i=1,2)$ of $M_t$ satisfy the following decay estimate
\begin{equation}
  |\lambda_i-1|\leq Ce^{-2^{(1-\alpha)}\cdot t},~~ i=1,2, ~~ \text{as }t\to\infty,
\end{equation}
where $C=C(\alpha,M_0)$ is a positive constant depending only on $\alpha$ and $M_0$. Moreover, the defining function $u$ of $M_t$ satisfies the following asymptotic behavior
\begin{equation}\label{thm1.2-f}
  u(t,\theta)=\frac t{2^{\alpha}}+f(\theta)+o(1),\quad \text{as}\quad t\to\infty,
\end{equation}
where $f(\theta)$ is a smooth function on $\mathbb{S}^2$.
\end{thm}

\begin{thm}\label{thm1.3}
For any smooth, closed strictly convex surface $M_0$ in $\mathbb{S}^3$,  there  exists a smooth
solution of flow \eqref{1.1} with $F$ satisfying  Assumption \ref{assum-1} and $\alpha=1$. The solution exists for finite time $t\leq T$ with $T<\infty$, $M_t$ expands to the equator as $t\to T$ and the properly rescaled surfaces converge exponentially in $C^{\infty}$-topology to the unit sphere $\mathbb{S}^2$.
\end{thm}

The expanding curvature flow for convex hypersurfaces driven by powers of a symmetric, increasing, homogeneous of degree one function of the principal curvatures has been studied by many authors. For $\alpha\in (0,1]$, Urbas \cite{Urbas91JDG} proved that the flow \eqref{1.1} in $\mathbb{R}^{n+1}$ exists for all time and converges to a round sphere after suitable rescaling, provided that either (i) the speed function $F$ is inverse concave and its dual function $F_*$ vanishes on the boundary of $\Gamma_+$, or (ii) $F$ is concave and inverse concave. For power $\alpha>1$, Gerhardt \cite{Gerhardt2014} proved the convergence of the flow \eqref{1.1} in $\mathbb{R}^{n+1}$ if $F$ is concave and $F$ vanishes on the boundary of $\Gamma_+$ (see earlier results by Schn\"{u}rer \cite{Schnurer2006} and Li \cite{Q-Li2010}); Kr\"{o}ner and Scheuer \cite{Kron-Sche2017} considered the case that $F$ is concave and initial hypersurface $M_0$ satisfies a suitable curvature pinching.  If the speed function $F$ is concave and $F$ vanishes on the boundary of $\Gamma_+$, Urbas \cite{urbas1990expansion} ($\alpha=1$) and Gerhardt \cite{Gerhardt1990,Gerhardt2014} ($\alpha\in (0,1]$) proved the long time existence and convergence of the flow \eqref{1.1} in $\mathbb{R}^{n+1}$, provided that the initial hypersurface is star-shaped (not necessarily convex) and admissible. The hyperbolic version was studied by Gerhardt \cite{gerhardt2011-H^n},  Scheuer \cite{Scheuer2015-1,scheuer2015-gradient} and the third author \cite{Wei-17}. Flow \eqref{1.1} for convex hypersurfaces in sphere was considered by Gerhardt \cite{Gerhardt2015}, Makowski-Scheuer \cite{makowski2013rigidity} and also by the third author \cite{Wei-17}. Inverse curvature flows have  been also studied in other Riemannian manifolds: in \cite{Sc2017,Sc2017-2,Zh2018} the ambient manifolds belong to a class of warped
products that includes the space forms, while in \cite{Pi1,Pi2} the ambient manifolds are non-compact rank one symmetric spaces and a different notion of mass is used to discuss the roundness of the limit of the rescaled metric.

In the previously mentioned papers, the concavity or the inverse concavity of $F$ plays an important role in deriving the curvature estimate and in applying the  second derivative H\"{o}lder estimates of Krylov \cite{Krylov}. Our Theorems \ref{thm1.1} -- \ref{thm1.3} say that the condition on the second derivatives of the speed function is not necessary  in the two-dimensional case. The proof is inspired by Andrews' work \cite{Andrews2010} on contracting curvature flow in $\mathbb{R}^3$. One ingredient of the proof is the pinching ratio estimate
\begin{equation*}
  \lambda_2~\leq~ C \lambda_1
\end{equation*}
along the flow \eqref{1.1} without any concavity assumption on the speed function $F$. This follows from applying the maximum principle to the evolution equation of
\begin{equation}\label{G-def}
  G=\frac{(\lambda_1-\lambda_2)^2}{(\lambda_1+\lambda_2)^2},
\end{equation}
which yields that the supremum of $G$ over $M_t$ is monotone decreasing in time along the flow \eqref{1.1} in $\mathbb{R}^3$ and $\mathbb{H}^3$ with $\alpha\in (0,1]$, and in $\mathbb{S}^3$ with $\alpha=1$. The reason that we can do this without condition on the second derivative of $F$ is that we can write down the gradient term completely in two-dimensional case and show that it has a favourable sign at the critical point of $G$ by using the condition $\nabla G=0$ at the critical point.

The curvature pinching estimate together with the bound on the speed $F$ will imply the bound on the principal curvatures as well as the preserving of the convexity of the evolving surfaces $M_t$. Since the speed function $F$ does not satisfy any second derivative condition, we use the second derivative H\"{o}lder estimate derived by Andrews \cite{andrews2004fully} in two-dimensional case and the standard parabolic Schauder estimate to derive the higher regularity of the flow. Then the long time existence of the flow \eqref{1.1} follows.

To prove the convergence of the flow, we need some extra work. In the Euclidean case, by using the curvature pinching estimate, we will refine the upper bound on $F$ and show that the properly rescaled $F$ is uniformly bounded from above. This together with the lower bound on the rescaled $F$ (obtained by Gerhardt \cite{Gerhardt2014}) implies the uniform two sides positive bounds on the rescaled principal curvatures. The Alexandrov reflection argument implies that the rescaled surfaces converge to a round sphere continuously. The uniform estimates on the rescaled principal curvatures and Andrews' \cite{andrews2004fully} second derivative H\"{o}lder estimate and the Schauder estimate can be applied to derive the smooth convergence of the rescaled flow. Finally, we can prove the exponential convergence of the rescaled flow. The method is to consider the quantity $G$ on the rescaled flow and show that the supremum of $G$ satisfies an exponential decay. The exponential decay of the curvature and the embedding of the flow surfaces then follow from a similar argument to that in \cite{Andrews1994-2}. Note that this exponential convergence of the flow \eqref{1.1} in Euclidean space was not considered in Gerhardt's paper \cite{Gerhardt2014}.

For the convergence of the flow in the hyperbolic space, the evolution equation of the quantity $G$ together with the contribution from the negative curvature of the hyperbolic case yields that the pinching ratio not only is controlled by its initial value but also decays to one exponentially. Thus in order to show that both principal curvatures $\lambda_1,\lambda_2$ converge to one as time goes to infinity, it suffices to show that the speed function $F$ converges to $2$ as time goes to infinity. For this, we adapt an argument used by Scheuer \cite{Scheuer2015-1}. Our pinching estimate will also be used in the proof. The remaining proof of Theorem \ref{thm1.2} follows from the similar argument as in \cite{Scheuer2015-1}.

Finally, for the flow \eqref{1.1} in the sphere, applying the argument in \cite{makowski2013rigidity} together with the pinching estimate and the regularity estimate of Andrews \cite{andrews2004fully}, we can first prove that the evolving surface $M_t$ expands to the equator in $C^{1,\gamma}$, $0<\gamma<1$ as $t\to T<\infty$. To show the smooth convergence we employ the dual flow method introduced by Gerhardt \cite{Gerhardt2015}. McCoy \cite{McCoy2017} proved that the dual contracting curvature flow of \eqref{1.1} contracts convex surface to a point in finite time and properly rescaled surfaces converge to a unit sphere smoothly. This could be used to show that the rescaled solution of the inverse curvature flow \eqref{1.1} converges to unit sphere continuously. By deriving the higher regularity estimate, we can eventually show that the rescaled solution converges to the unit sphere smoothly.

In Theorem \ref{thm1.2}, the convergence of the flow states that the principal curvatures converge to $1$ exponentially and the shifted defining function of the flow surface converges to a smooth function $f(\theta)$ over the sphere $S^2$ as time goes to infinity. In general, the function $f(\theta)$ may not be a constant function nor a first eigenfunction of the Laplacian operator on $S^2$. This means that the conformal metric $e^{2f}g_{S^2}$ may not be a round metric on the sphere. This property for the inverse curvature flow in hyperbolic space cannot be improved as we prove in the following theorem. \begin{thm}\label{thm-ex-H}
There exists a smooth star-shaped and mean convex surface $M_0$ in $\mathbb{H}^3$ such that along flow \eqref{1.1} with $\alpha\in(0,1), F=H$, as $t\to\infty$, the rescaled metric $e^{-2^{(1-\alpha)}\cdot t}g_t$ converges to a metric on $\mathbb{S}^2$ which is not the round metric, where $g_t$ is the induced metric on $M_t$.
\end{thm}
An analogous counterexample for the inverse mean curvature flow (i.e., flow \eqref{1.1} with $\alpha=1, F=H$) in $\mathbb{H}^3$ was constructed by Hung-Wang \cite{Hung-Wang2015inverse}. The proof of our Theorem \ref{thm-ex-H} is inspired by their argument. Consider the following quantity for any smooth surface $M$ in $\mathbb{H}^3$
\begin{equation}\label{def-m}
  Q(M)=-|M|\int_{M}|\mathring{A}|^2d\mu,
\end{equation}
where $\mathring{A}$ is the trace-less part of the second fundamental form of $M$. Let $\tilde{M}_s$ be a family of surfaces in $\mathbb{H}^3$ that are radial graphs of the functions $u(s,\theta)=cs+f(\theta)+o(1)$ over the sphere $S^2$ and $g_{ij}$ be the induced metric on $\tilde{M}_s$. Hung-Wang proved that the limit of the rescaled metric $e^{-2cs}g_{ij}$ as $s\to \infty$ is a round metric if and only if $\lim_{s\to\infty}Q(\tilde{M_s})=0$. This characterizes when the limit of rescaled metric is a round metric in terms of the value of $Q$.

To construct the example in Theorem \ref{thm-ex-H}, we choose a function $\bar{f}(\theta)$ on $\mathbb{S}^2$ such that
\begin{equation*}
  \int_{\mathbb{S}^2}e^{2\bar{f}}d\mu_{g_{\mathbb{S}^2}}\int_{\mathbb{S}^2}|\mathring{D}^2e^{-\bar{f}}|_{g_{\mathbb{S}^2}}^2d\mu_{g_{\mathbb{S}^2}}=c_0>0,
\end{equation*}
where $\mathring{D}^2e^{-\bar{f}}$ means the traceless part of the Hessian of $e^{-\bar{f}}$.
Let $\tilde{M}_s$ be the family of surfaces given by the radial graph of $u(s,\theta)=s+\bar{f}(\theta)$ over $\mathbb{S}^2$ in the $(r,\theta)$ coordinates of $\mathbb{H}^3$. We know from \cite{Hung-Wang2015inverse} that
\begin{equation}
  \lim_{s\to\infty}Q(\tilde{M}_s)=-c_0.
\end{equation}
By choosing $s_0>0$ large enough and using some estimate from Neves' paper \cite{Neves-2010}, we have that for any $s\geq s_0$, $\tilde{M}_s$ is mean-convex,  star-shaped and is strongly pinched. We consider the solution $M_t^s$ of the flow \eqref{flow-H^k} starting from $\tilde{M}_s, s\geq s_0$, where $t$ is the time parameter. We will complete the proof of Theorem \ref{thm-ex-H} in \S \ref{sec:HW} by estimating the limit
\begin{equation}\label{s1:est-1}
  \lim_{t\to\infty} Q(M_t^s)~\leq ~-c_0/4<0
\end{equation}
for $s\geq s_0$, where $s_0$ is chosen large enough. To obtain \eqref{s1:est-1}, we need the following crucial estimate
\begin{equation}\label{s1:est-decay}
\begin{aligned}
  |\tilde{M}_s|^2|H-2|^2+|\tilde{M}_s|^2|\mathring{A}|^2+|\tilde{M}_s|^2|\nabla A|^{4/3} ~\leq~&Ce^{-2^{(2-\alpha)}\cdot t},
\end{aligned}
\end{equation}
where $H$, $A$  and $\mathring{A}$ are respectively mean curvature, the second fundamental form and the trace-less part of the second fundamental form of the solution $M_t^s$ of the flow \eqref{flow-H^k} starting from $\tilde{M}_s$ for any $s\geq s_0$. The key point of the estimate \eqref{s1:est-decay} is that the constant $C$ is independent of the parameter $s$ so that we can obtain the desired estimate \eqref{s1:est-1} by properly choosing $s_0$ large enough. The proof of \eqref{s1:est-decay} is the central part of \S \ref{sec:HW} and is technical. The idea is from Neves \cite{Neves-2010} in the case $\alpha=1$. However, for $\alpha\in (0,1)$, the flow \eqref{1.1} is not scale invariant, several technique difficulties will arise in obtaining the sharp exponential decay estimate \eqref{s1:est-decay}. Our strategy is by first obtaining a weaker decay estimate and then improving the rate step by step. See  \S \ref{sec:7-1} for details.

\begin{ack}
The first author was supported in part by NSFC Grant No. 11671214.
The second author was supported in part by NSFC Grant No. 11571185. The third author was supported by Ben Andrews throughout his Australian Laureate Fellowship FL150100126 of the Australian Research Council. The authors would like to thank Ben Andrews and Julian Scheuer for comments on the earlier version of this paper, and the referees for carefully reading of this paper and providing many helpful suggestions.
\end{ack}

\section{Notations and preliminaries}
In this section, we give some notations and preliminary results.
Throughout the paper, we adopt the Einstein summation convention of summing over repeated indices.
Let $X_t=X(x,t):M_t\to\mathbb{R}^3(\kappa)$ be a family of surfaces in a real space form $\mathbb{R}^3(\kappa)$ moving according to
\begin{equation}\label{flow1}
  \frac{\partial}{\partial t}X(x,t)=-\Phi(x,t)\nu(x,t),
\end{equation}
 where $\Phi(x,t)=\Phi(F(x,t))=-F^{-\alpha}(x,t)$, $F$ is a smooth, symmetric, and homogeneous of degree one function of the principal curvatures of the surfaces $M_t=X_t(M)$ and $\nu$ is the outer unit normal of $M_t$.

We  use  $g=\{g_{ij}\}$, $A=\{h_{ij}\}$ and $\mathcal{W}=\{h^i_j\}$ to denote  the components of  induced metric,  the  second fundamental form and the Weingarten map of the surfaces, respectively. Here $h_{ij}=-\bar{g}(\bar{\nabla}_{\frac{\partial X}{\partial x^i}}\frac{\partial X}{\partial x^j},\nu)$ in local coordinates $x^1,x^2$, where $\bar{g}$ denotes the metric of $\mathbb{R}^3(\kappa)$ and $\bar{\nabla}$ denotes the Levi-Civita connection with respect to the metric $\bar{g}$. Note that at a given point $x\in M$, we can always choose local coordinates $x^1,x^2$ such that $g_{ij}=\delta_{ij},~\nabla_{\frac{\partial}{\partial x^i}}\frac{\partial}{\partial x^j}=0$ and $h^i_j=\text{diag}(\lambda_1,\lambda_2)$ at $x$.
The function $F=F(\mathcal{W})=F(\lambda(\mathcal{W}))$ can be considered as a function of $\mathcal{W}=(h_i^j)$ or the principal curvatures  $\lambda(\mathcal{W})=(\lambda_1,\lambda_2)$. The derivatives of $F$ with respect to $\lambda_i$ and $h_i^j$ are related in the following way (see \cite{Andrews1994-2,Andrews2007,Gerhardt1996}). If $A$ is a diagonal and $B$ a symmetric matrix, then
\begin{equation*}
  \dot{F}^{kl}(A)=~\frac{\partial F}{\partial \lambda_k}(\lambda(A))\delta^{kl},
\end{equation*}
and if $A$ has distinct eigenvalues, then the second derivative of $F$ in direction $B$ is given by
\begin{equation}\label{2.5}
  \ddot{F}^{kl,rs}(A)B_{kl}B_{rs}=~\frac{\partial^2F(\lambda(A))}{\partial\lambda_k\partial\lambda_l}B_{kk}B_{ll}+2\sum_{k<l}\frac{\frac{\partial F}{\partial \lambda_k}-\frac{\partial F}{\partial \lambda_l}}{\lambda_k-\lambda_l}B_{kl}^2.
\end{equation}
The second term makes sense as a limit if $\lambda_k=\lambda_l$. Since $F$ is symmetric, we may assume that at each point $(x,t)\in M\times [0,T)$, the principal curvatures satisfy $\lambda_1\leq \lambda_2$.

\subsection{Evolution equations}
For the surfaces moving according to \eqref{flow1}, we have the following evolution equations (see   \cite{Andrews1994-3,Gerhardt2006}):
\begin{equation}\label{s2-nu-evl}
  \frac{\partial}{\partial t}\nu=\nabla\Phi,
\end{equation}
\begin{equation}\label{2.2}
\frac{\partial}{\partial t}g_{ij}=-2\Phi h_{ij},
\end{equation}
\begin{equation}\label{2.3}
\frac{\partial}{\partial t}\Phi-\dot{\Phi}\dot{F}^{ij}\nabla_i\nabla_j\Phi=\Phi\dot{\Phi}\dot{F}^{ij}h_{ip}h^p_j+\kappa \Phi\dot{\Phi}\dot{F}^{ij}g_{ij},
\end{equation}
\begin{equation}\label{2.4}
\begin{aligned}
\frac{\partial}{\partial t}h_{ij}-\dot{\Phi}\dot{F}^{kl}\nabla_k\nabla_lh_{ij}=&\dot{\Phi}\ddot{F}^{kl,mn}\nabla_ih_{kl}\nabla_jh_{mn}+\ddot{\Phi}\dot{F}^{kl}\nabla_ih_{kl}\dot{F}^{mn}\nabla_jh_{mn}\\
&+\dot{\Phi}\dot{F}^{kl}h_{ij}h^p_kh_{pl}-\dot{\Phi}Fh^k_ih_{kj}-\Phi h^k_ih_{kj}\\
&+\kappa(\dot{\Phi}Fg_{ij}+\Phi g_{ij}
-\dot{\Phi}\dot{F}^{kl}g_{kl}h_{ij}),
\end{aligned}
\end{equation}
where $\nabla$ is the Levi-Civita connection with respect to the induced metric $g$, $\nabla_ih_{kl}$ is the covariant derivative of the second fundamental form, and $\dot{\Phi}={d\Phi}/{dr}$ by considering $\Phi=\Phi(r)=-r^{-\alpha}$ as a function of the real variable $r$. For any function $\omega$ on $M$, we also use $\omega_i=\nabla_i\omega$, $\omega_{ij}=\nabla_i\nabla_j\omega$ to denote the covariant derivatives of $\omega$ with respect to $g_{ij}$.

\subsection{Graphical representation}\label{sec:2-2}
We recall the warped product model of the space form $\mathbb{R}^3(\kappa)$, i.e., $\mathbb{R}^3(\kappa)=I\times\mathbb{S}^2$ equipped with the warped product metric
\begin{equation*}
  \bar{g}=dr^2+s_{\kappa}^2(r)g_{\mathbb{S}^2},
\end{equation*}
where $I=(0,\infty)$ if $\kappa=0,-1$, and $I=(0,\pi) $ if $\kappa=1$ with
\begin{equation*}
 s_{\kappa}(r)=\left\{\begin{aligned}
  r,~~& \quad\kappa=0,\\
  \sin r, &\quad\kappa=1,\\
  \sinh r, &\quad \kappa=-1.
  \end{aligned}\right.
\end{equation*}
Suppose that $M$ is a star-shaped surface in $\mathbb{R}^3(\kappa)$ and can be expressed as a graph over the sphere $\mathbb{S}^2$, i.e., $ M=\{(u(\theta),\theta),~ \theta\in\mathbb{S}^2\}$ for some function $u\in C^{\infty}(\mathbb{S}^2)$, then the induced metric on $M$ in terms of the coordinates $\theta^j$ is given by
\begin{equation}\label{s2:gij}
  g_{ij}=u_iu_j+s^2_{\kappa}(u)\sigma_{ij},
\end{equation}
where $\sigma_{ij}=g_{\mathbb{S}^2}(\partial_{\theta^i},\partial_{\theta^j})$ are the components of the round metric $g_{\mathbb{S}^2}$.
The second fundamental form $h_{ij}$ satisfies
\begin{equation}\label{s2:hij}
  h_{ij}v^{-1}=-u_{ij}+s'_{\kappa}(u)s_{\kappa}(u)\sigma_{ij},
\end{equation}
where $u_{ij}$ are the covariant derivatives of $u$ with respect to the induced metric $g_{ij}$, $s'_{\kappa}(r)$ is the derivative of $s_{\kappa}(r)$ and $v$ is defined by
\begin{equation}\label{s2:v-def}
  v=\sqrt{1+s_{\kappa}^{-2}(u)|Du|^2_{g_{\mathbb{S}^2}}}.
\end{equation}
The unit normal vector field on the surface is given by
\begin{equation}\label{normal-unit}
  \nu=v^{-1}(\partial_r-s_{\kappa}^{-2}(u)u^j\partial_{\theta^j}).
\end{equation}
We denote
 \begin{equation*}
   \varphi(u)=\int_{u_0}^u \frac 1{s_{\kappa}(r)}dr,
 \end{equation*}
then $\varphi'(u)=s^{-1}_{\kappa}(u)$ and $h_j^i$ (the components of the  Weingarten map) can be expressed as
\begin{equation}\label{hij-warp}
  h_j^i=v^{-1}s^{-1}_{\kappa}(u)\left(-(\sigma^{ik}-v^{-2}\varphi^i\varphi^k)\varphi_{jk}+s'_{\kappa}(u)\delta_j^i\right),
\end{equation}
where $\varphi^i=\sigma^{ik}\varphi_k$, $(\sigma^{ij})=(\sigma_{ij})^{-1}$ and the covariant derivatives here are taken with respect to $\sigma_{ij}$.

If $M_t$ is a smooth star-shaped solution of \eqref{flow1} for $t\in[0,T)$ and each flow surface is expressed as a graph $M_t=\text{graph } u(t,\theta)$ over the sphere $\mathbb{S}^2$, we can easily deduce that the defining function $u(t,\theta)$ of $M_t$ satisfies the following scalar parabolic equation (see \cite{Gerhardt2006})
\begin{equation}\label{s2:ut-evl}
  \frac{\partial}{\partial t} u(t)=-\Phi v,
\end{equation}
on $[0,T)\times \mathbb{S}^2$, where $v$ is the function defined in \eqref{s2:v-def}.

\subsection{Support function}\label{sec:2-3}
The support function of a star-shaped surface $M$ in $\mathbb{R}^3(\kappa)$ is defined by
\begin{equation*}
  \chi=\bar{g}( s_{\kappa}(u)\partial_r,\nu)=s_{\kappa}(u)v^{-1}.
\end{equation*}
In this subsection, we derive the evolution equation of $\chi$ for the flow surfaces $M_t$ along the flow \eqref{flow1}.  We choose local coordinates $x^1,x^2$ such that $g_{ij}=\delta_{ij},~\nabla_{\frac{\partial}{\partial x^i}}\frac{\partial}{\partial x^j}=0$ and $h^i_j=\text{diag}(\lambda_1,\lambda_2)$ at a given point $x\in M$. It is easy to check that the vector field $s_{\kappa}(u)\partial_r$ is a conformal vector field in the sense that (cf. \cite[page 206]{ONEILL})
\begin{equation}\label{s2:conf}
  \bar{\nabla}_Y(s_{\kappa}(u)\partial_r)=s'_{\kappa}(u)Y
\end{equation}
for any tangent vector field $Y$ on $\mathbb{R}^3(\kappa)$, where $\bar{\nabla}$ denotes the Levi-Civita connection with respect to the metric $\bar{g}$ on $\mathbb{R}^3(\kappa)$. Then using \eqref{s2:conf} and \eqref{s2-nu-evl}, we have
\begin{align}\label{s2:chi-dt}
  \frac{\partial}{\partial t}\chi =&\bar{g}( \bar{\nabla}_{-\Phi\nu}(s_{\kappa}(u)\partial_r),\nu)+\bar{g}( s_{\kappa}(u)\partial_r, \partial_t\nu) \nonumber \\
  = & -\Phi s'_{\kappa}(u)+ \bar{g}( s_{\kappa}(u)\partial_r, \nabla\Phi).
\end{align}
Similarly, using \eqref{s2:conf}, we derive that
\begin{align}
\nabla_i\chi=&\bar{g}( s_{\kappa}(u)\partial_r, h_i^k\partial_{x^k}),\label{s2:chi-ds1}\\
  \nabla_j\nabla_i\chi= & \bar{g}( s_{\kappa}(u)\partial_r,g^{kl}\nabla_lh_{ij}\partial_{x^k})+h_{ij}s'_{\kappa}(u)-h_i^kh_{kj}\chi.\label{s2:chi-ds2}
\end{align}
Combining the equations \eqref{s2:chi-dt}-\eqref{s2:chi-ds2} gives that
\begin{equation}\label{s2:evl-chi}
  \frac{\partial}{\partial t}\chi-\dot{\Phi}\dot{F}^{ij} \nabla_j\nabla_i\chi=\dot{\Phi}\dot{F}^{ij}h_i^kh_{kj}\chi-(\Phi+\dot{\Phi}F)s'_{\kappa}(u).
\end{equation}
Note that
\begin{equation*}
  -(\Phi+\dot{\Phi}F)s'_{\kappa}(u)=(1-\alpha)F^{-\alpha}s'_{\kappa}(u)\geq 0
\end{equation*}
for $\alpha\in(0,1]$. The maximum principle applied to \eqref{s2:evl-chi} then implies that the star-shaped condition $\chi>0$ of the flow surfaces is preserved along the flow \eqref{flow1}.

\section{Estimate on the pinching ratio}\label{sec:3}
In this section, we  show that the pinching ratio of $M_t$, which is the supremum over the surface $M_t$ of the ratio of largest to smallest principal curvatures at each point,  is no greater than that of $M_0$. This is the first key step of the proof of Theorems \ref{thm1.1} -- \ref{thm1.3}. The idea is to apply the maximum principle to the evolution equation for the following quantity $G$.

\begin{thm}\label{thm3.1}
Let  $M_t~(0\leq t<T)$ be a family of smooth, closed strictly convex surfaces in $\mathbb{R}^3(\kappa)(\kappa=0,1,-1)$ flowing according to  \eqref{1.1} with $F$ satisfying  Assumption \ref{assum-1}. We assume that $\alpha\in (0,1]$ in the cases $\kappa=0,-1$, and assume that $\alpha=1$ in the case $\kappa=1$. Then we have that the supremum $\tilde{G}$ of the quantity $G=\frac{(\lambda_1-\lambda_2)^2}{(\lambda_1+\lambda_2)^2}$ is
non-increasing in time.
\end{thm}
\begin{proof}
Denote $\Phi=-F^{-\alpha}$. Then $\Phi$ is a symmetric homogeneous of degree $-\alpha$ function of the principal curvatures of $M_t$. By a direct computation, we have (see \cite{Andrews1994-3})
\begin{equation}\label{3.2}
\begin{aligned}
\frac{\partial}{\partial t}G-\dot{\Phi}^{ij}\nabla_i\nabla_jG=&(\dot{G}^{ij}\ddot{\Phi}^{kl,mn}-\dot{\Phi}^{ij}\ddot{G}^{kl,mn})\nabla_ih_{kl}\nabla_jh_{mn}\\
&-\dot{G}^{ij}\dot{\Phi}^{kl}h_{kl}h^p_ih_{pj}+\dot{G}^{ij}\dot{\Phi}^{kl}h_{ij}h^p_kh_{pl}+\Phi\dot{G}^{ij}h^k_ih_{kj}\\
&+\kappa(\dot{G}^{ij}\dot{\Phi}^{kl}h_{kl}g_{ij}
-\dot{G}^{ij}\dot{\Phi}^{kl}g_{kl}h_{ij}+\Phi\dot{G}^{ij}g_{ij}),
\end{aligned}
\end{equation}
where $\dot{\Phi}^{ij}=\dfrac{\partial \Phi}{\partial h_{ij}}$, $\ddot{\Phi}^{kl,mn}=\dfrac{\partial^2\Phi}{\partial h_{kl}\partial h_{mn}}$. Note that $G$ is homogeneous of degree zero, $\Phi$ is  homogeneous of degree $-\alpha$, and the Euler relation gives that $\dot{G}^{ij}h_{ij}=0$ and $\dot{\Phi}^{ij}h_{ij}=-\alpha\Phi$, so \eqref{3.2} can be simplified by
\begin{align}\label{3.2_1}
\frac{\partial}{\partial t}G-\dot{\Phi}^{ij}\nabla_i\nabla_jG=&(\dot{G}^{ij}\ddot{\Phi}^{kl,mn}-\dot{\Phi}^{ij}\ddot{G}^{kl,mn})\nabla_ih_{kl}\nabla_jh_{mn}\nonumber\\
&+(1+\alpha)\Phi\dot{G}^{ij}h^k_ih_{kj}+\kappa(1-\alpha)\Phi\dot{G}^{ij}g_{ij}.
\end{align}
We denote the zero-order terms and the first-order terms  in \eqref{3.2_1} by $Q_0$ and $Q_1$, respectively:
\begin{equation*}
\begin{aligned}
Q_0&=(1+\alpha)\Phi\dot{G}^{ij}h^k_ih_{kj}+\kappa(1-\alpha)\Phi\dot{G}^{ij}g_{ij},\\
Q_1&=(\dot{G}^{ij}\ddot{\Phi}^{kl,mn}-\dot{\Phi}^{ij}\ddot{G}^{kl,mn})\nabla_ih_{kl}\nabla_jh_{mn}.
\end{aligned}
\end{equation*}

In the following,  at a given point $x_t\in M_t$, we will choose  local coordinates $x^1,x^2$ such that $g_{ij}=\delta_{ij},~\nabla_{\frac{\partial}{\partial x^i}}\frac{\partial}{\partial x^j}=0$ and $h^i_j=\text{diag}(\lambda_1,\lambda_2)$ at $x_t$. Thus we can use \eqref{2.5} to simplify $Q_0$ and $Q_1$ as follows:
\begin{equation}\label{3.4}
\begin{aligned}
Q_0=&(1+\alpha)\Phi(\frac{\partial G}{\partial \lambda_1}\lambda_1^2+\frac{\partial G}{\partial \lambda_2}\lambda_2^2)+\kappa(1-\alpha)\Phi(\frac{\partial G}{\partial \lambda_1}+\frac{\partial G}{\partial \lambda_2}),
\end{aligned}
\end{equation}
\begin{equation}\label{3.5}
\begin{aligned}
Q_1=&(\frac{\partial G}{\partial \lambda_1}\frac{\partial^2 \Phi}{\partial \lambda_1^2}- \frac{\partial \Phi}{\partial \lambda_1}\frac{\partial^2 G}{\partial \lambda_1^2}) (\nabla_1h_{11})^2 + (\frac{\partial G}{\partial \lambda_1}\frac{\partial^2 \Phi}{\partial \lambda_2^2}- \frac{\partial \Phi}{\partial \lambda_1}\frac{\partial^2 G}{\partial \lambda_2^2}) (\nabla_1h_{22})^2\\
&+(\frac{\partial G}{\partial \lambda_2}\frac{\partial^2 \Phi}{\partial \lambda_1^2}- \frac{\partial \Phi}{\partial \lambda_2}\frac{\partial^2 G}{\partial \lambda_1^2}) (\nabla_2h_{11})^2 + (\frac{\partial G}{\partial \lambda_2}\frac{\partial^2 \Phi}{\partial \lambda_2^2}- \frac{\partial \Phi}{\partial \lambda_2}\frac{\partial^2 G}{\partial \lambda_2^2}) (\nabla_2h_{22})^2\\
&+2 (\frac{\partial G}{\partial \lambda_1}\frac{\partial^2 \Phi}{\partial \lambda_1\partial \lambda_2}-\frac{\partial \Phi}{\partial \lambda_1}\frac{\partial^2 G}{\partial \lambda_1\partial \lambda_2})\nabla_1h_{11}\nabla_1h_{22}\\
&+2 (\frac{\partial G}{\partial \lambda_2}\frac{\partial^2 \Phi}{\partial \lambda_1\partial \lambda_2}-\frac{\partial \Phi}{\partial \lambda_2}\frac{\partial^2 G}{\partial \lambda_1\partial \lambda_2})\nabla_2h_{11}\nabla_2h_{22}\\
 &+2\frac{\frac{\partial G}{\partial \lambda_1}\frac{\partial \Phi}{\partial \lambda_2}-\frac{\partial G}{\partial \lambda_2}\frac{\partial \Phi}{\partial \lambda_1}}{\lambda_2-\lambda_1}(\nabla_1h_{12})^2+2\frac{\frac{\partial G}{\partial \lambda_1}\frac{\partial \Phi}{\partial \lambda_2}-\frac{\partial G}{\partial \lambda_2}\frac{\partial \Phi}{\partial \lambda_1}}{\lambda_2-\lambda_1}(\nabla_2h_{12})^2.
\end{aligned}
\end{equation}
By the definition of $G$ (see \eqref{G-def}), we have
\begin{equation}\label{partialG}
\frac{\partial G}{\partial \lambda_1}=\frac{4\lambda_2(\lambda_1-\lambda_2)}{(\lambda_1+\lambda_2)^3},~\frac{\partial G}{\partial \lambda_2}=\frac{4\lambda_1(\lambda_2-\lambda_1)}{(\lambda_1+\lambda_2)^3}.
\end{equation}
Using \eqref{partialG}, we can simplify \eqref{3.4}:
\begin{equation}\label{Q0}
Q_0=\frac{4\Phi G}{\lambda_1+\lambda_2}\Big(\kappa(\alpha-1)+(1+\alpha)\lambda_1\lambda_2\Big).
\end{equation}

For each time $t>0$, at the critical point of $G$ with $\lambda_1\neq \lambda_2$, from \eqref{partialG}, we have $\frac{\partial G}{\partial \lambda_i}\neq 0$, and the gradient condition on $G$, i.e., $0=\nabla_iG= \frac{\partial G}{\partial \lambda_1}\nabla_ih_{11}+\frac{\partial G}{\partial \lambda_2}\nabla_ih_{22},~i=1,2$, leads
to the following two equations:
\begin{equation}\label{3.8}
\nabla_1h_{11}=-\frac{\frac{\partial G}{\partial \lambda_2}}{\frac{\partial G}{\partial \lambda_1}}\nabla_1h_{22}=\frac{\lambda_1}{\lambda_2}\nabla_1h_{22},~\nabla_2h_{22}=-\frac{\frac{\partial G}{\partial \lambda_1}}{\frac{\partial G}{\partial \lambda_2}}\nabla_2h_{11}=\frac{\lambda_2}{\lambda_1}\nabla_2h_{11}.
\end{equation}
Note that the Codazzi equations  say that $\nabla_kh_{ij}$ is totally symmetric, i.e., we have $\nabla_1h_{12}=\nabla_2h_{11}$ and $\nabla_2h_{12}=\nabla_1h_{22}$.
Using \eqref{3.8} and the homogeneity of $G$ and $\Phi$, we arrive at
\begin{equation}\label{Q1}
Q_1=\frac{4\alpha \Phi}{(\lambda_1+\lambda_2)^3}\Big(\big((1+\alpha)\frac{\lambda_1}{\lambda_2}+(1-\alpha)\big)
(\nabla_1h_{22})^2+\big((1-\alpha)+(1+\alpha)\frac{\lambda_2}{\lambda_1}\big)(\nabla_2h_{11})^2\Big).
\end{equation}
In view of \eqref{Q0} and \eqref{Q1}, we have $Q_0\leq 0$ and $Q_1\leq 0$ in the cases $\alpha\in (0,1]$ for $\kappa=0,-1$, and $\alpha=1$ for $\kappa=1$. Applying the maximum principle, we conclude that the supremum of $G$ on $M_t$ is non-increasing in time $t$.
\end{proof}

Since the pinching ratio is given by $r=\frac{2}{1-\sqrt{\tilde{G}}}-1$, as a corollary of Theorem \ref{thm3.1}, we have
\begin{prop}\label{prop3.2}
Let  $M_t~(0\leq t<T)$ be a family of smooth closed strictly convex surfaces in $\mathbb{R}^3(\kappa)(\kappa=0,1,-1)$ flowing according to  \eqref{1.1} with $F$ satisfying  Assumption \ref{assum-1}. We assume that $\alpha\in (0,1]$ for $ \kappa=0,-1,$ and $\alpha=1$ for $\kappa=1$. Then the
pinching ratio of $M_t$ is no greater than the pinching ratio of $M_0$, i.e, there exists a positive constant $\beta>1$ which only depends on $M_0$ such that
\begin{equation}\label{3.10}
\lambda_2\leq\beta \lambda_1.
\end{equation}
\end{prop}
\begin{rem}
Since $F$ is homogeneous of degree one, we have that $\partial F/{\partial\lambda_i}, i=1,2$ are homogeneous of degree zero, and then we have that
\begin{equation*}
  \frac{\partial F}{\partial \lambda_i}(\lambda_1,\lambda_2)=\frac{\partial F}{\partial \lambda_i}(\frac{\lambda_1}{\lambda_1+\lambda_2},\frac{\lambda_2}{\lambda_1+\lambda_2}).
\end{equation*}
By the pinching ratio estimate in Proposition \ref{prop3.2}, the supremum and infimum of $\partial F/{\partial\lambda_i}, i=1,2$ are attained on the compact set $\{(a,1-a)): |a-\frac 12|\leq \frac{\beta-1}{2(\beta+1)}\}$ and hence there exists a positive constant $c$ depending only on $\beta$ (and therefore depending only on $M_0$) such that
\begin{equation}\label{pt-F-bd1}
  0<c^{-1}\leq \frac{\partial F}{\partial \lambda_i}\leq c_,\quad i=1,2,\quad \forall~t\in[0,T).
\end{equation}
\end{rem}

\begin{rem}
We also have the following consequences of the pinching ratio estimate \eqref{3.10}.  As $F$ is homogeneous of degree one, normalized with $F(1,1)=2$ and is strictly monotone in each argument, we have
\begin{align*}
\frac{2\lambda_2}{\beta}\leq2\lambda_1=F(\lambda_1,\lambda_1)\leq F(\lambda_1,\lambda_2)\leq F(\lambda_2,\lambda_2)=2\lambda_2\leq 2\beta\lambda_1,
\end{align*}
which is equivalent to
\begin{align}\label{F-lambda}
\frac{F}{2\beta}\leq \lambda_1\leq\frac{F}{2}\leq \lambda_2\leq\frac{\beta F}{2}.
\end{align}
So we have
\begin{align}\label{F-bd-H-pf2}
  \dot{F}^{ij}g_{ij}= \frac{\partial F}{\partial\lambda_1}+ \frac{\partial F}{\partial\lambda_2} ~\geq ~\frac 1{\lambda_2} \left(\frac{\partial F}{\partial\lambda_1}\lambda_1+ \frac{\partial F}{\partial\lambda_2}\lambda_2\right)=\frac {F(\lambda_1,\lambda_2)}{\lambda_2}\geq \frac 2\beta,
\end{align}
and
\begin{align}\label{3.12}
  \dot{F}^{ij}h_i^kh_{kj}=\frac{\partial F}{\partial \lambda_i}\lambda_i^2\geq  & \left(\frac{\partial F}{\partial \lambda_1}\lambda_1+\frac{\partial F}{\partial \lambda_2}\lambda_2\right)\lambda_1\geq \frac {\lambda_2}{\beta}F\geq \frac 1{2\beta}F^2.
\end{align}
Similarly, we have
\begin{equation}\label{3.13}
  \dot{F}^{ij}g_{ij}\leq 2\beta,\quad \text{and}\quad  \dot{F}^{ij}h_i^kh_{kj}\leq \frac{\beta}2F^2.
\end{equation}
\end{rem}
\begin{rem}
We further remark that the pinching ratio estimate in Proposition \ref{prop3.2} cannot be improved. For $\alpha>1$, after a similar argument to that in \cite[\S 5]{Andrews2010}, we can obtain an example of smooth, strictly convex surface in $\mathbb{R}^3$ for which the pinching ratio becomes larger for any flow \eqref{1.1} with  $\alpha>1$. Moreover, Kr\"{o}ner and Scheuer \cite{Kron-Sche2017} constructed a counterexample to show that along the flow \eqref{1.1} in $\mathbb{R}^3$ with $F=H$ and $\alpha>1$, the convexity of the initial surface will be lost. However, the fact that the pinching ratio does not improve obviously does not rule out the possibility that other curvature estimates may yield useful results, cf. the results in \cite{Gerhardt2014,Scheuer2015-1,Kron-Sche2017} for $\alpha>1$ case.
\end{rem}

\section{Flow in Euclidean space}
In this section,  we consider flow \eqref{1.1} in Euclidean space $\mathbb{R}^3$ and prove Theorem \ref{thm1.1}. We assume that $F$ satisfies  Assumption \ref{assum-1} and $\alpha\in(0,1]$. Since the flow \eqref{1.1} is a parabolic equation with strictly convex initial data, by short time existence theorem we have a smooth solution on a maximal time interval $[0,T),~0<T\leq\infty.$ It remains to study the long time behavior of the flow.  In \cite{Gerhardt2014}, the concavity of $F$ is essentially used in order to get the curvature estimates. Here, we do not have the concavity assumption and we use
the pinching ratio estimate obtained in Section 3 to prove the curvature estimates and the convergence of the flow. We also show that the rescaled flow converges exponentially to the sphere.

The following lemma gives the evolution of spheres in $\mathbb{R}^3$ along the flow.
\begin{lem}\label{eg}
Given $x_0\in \mathbb{R}^3, \rho_0\in \mathbb{R}^+$. Denote
\begin{equation}\label{rho_t}
  \rho(t,\rho_0)=\left\{\begin{array}{ll}
  \left((1-\alpha)2^{-\alpha}t+\rho_0^{1-\alpha}\right)^{\frac{1}{1-\alpha}}, & \alpha\in(0,1) \\
                   \rho_0e^{t/2},& \alpha=1.
                 \end{array}\right.
\end{equation}
Then the spheres $\partial B_{\rho(t,\rho_0)}(x_0)$ solve \eqref{1.1} for $t\in[0,+\infty)$ with $\rho_0$ as the initial radius.
\end{lem}
\begin{proof}
Since the flow \eqref{1.1} preserves the symmetry, in the sphere case, the equation \eqref{1.1} reduces to the following ODE for the radius of the spheres
\begin{equation}\label{4.2}
\left\{\begin{aligned}
\frac{d}{dt}\rho(t,\rho_0)=&~2^{-\alpha}\rho(t,\rho_0)^{\alpha},\\
\rho(0,\rho_0)=&~\rho_0.
\end{aligned}\right.
\end{equation}
Then the lemma follows immediately by solving \eqref{4.2}.
\end{proof}

Since the initial surface $M_0$ is assumed to be closed and strictly convex in $\mathbb{R}^3$, we can choose a point $x_0$ in the domain enclosed by $M_0$ such that $M_0$ is given by the graph of a smooth function $u_0$ over the sphere $\mathbb{S}^2(x_0)=\partial B_1(x_0)$. Without loss of generality, we can assume that $x_0$ is the origin and denote $\mathbb{S}^2(x_0)$ simply by $\mathbb{S}^2$. Then $M_0=\{(u_0(\theta),\theta), \theta\in\mathbb{S}^2\}$. Under the assumption of Theorem \ref{thm1.1}, the flow \eqref{1.1} preserves the star-shaped condition (see Section \ref{sec:2-3}) and each flow surface $M_t$ can be written as a graph of $u(t,\theta)$ over the sphere. By \eqref{s2:ut-evl}, the defining function $u$ satisfies the following scalar flow equation
\begin{equation}\label{flow-u}
  \frac{\partial u}{\partial t}=\sqrt{1+|D\log u|^2_{g_{\mathbb{S}^2}}}~F^{-\alpha},
\end{equation}
which is clearly parabolic.

We assume that the initial surface $M_0= \text{graph} ~ u_0$ satisfies
\begin{equation*}
  \rho_1<u_0(\theta)<\rho_2,\quad \forall~\theta\in\mathbb{S}^2,
\end{equation*}
where $\rho_1$ and $\rho_2$ are two positive constants.
By applying the maximum principle, we have
\begin{lem}[\cite{Gerhardt2014}]
As long as the flow \eqref{1.1} exists, the defining function $u(t,\theta)$ of the solution of the flow satisfies
\begin{equation}\label{ut}
  \rho(t,\rho_1)<u(t,\theta)<\rho(t,\rho_2), \quad ~\forall~t\in [0,T),~\forall~\theta\in\mathbb{S}^2.
\end{equation}
Moreover, for $\rho_1<\bar{r}<\rho_2$, there exist two positive constants $c_1,c_2$ depending only on $\rho_1,\rho_2$ and $\alpha$ such that
\begin{equation}\label{ut2}
  0<c_1\leq u(t,\theta)\rho^{-1}(t,\bar{r})\leq c_2,\quad \forall ~t\in [0,T),~\forall~\theta\in\mathbb{S}^2.
\end{equation}
In particular, the flow is compactly contained in $\mathbb{R}^3$ for finite time $t$.
\end{lem}

Moreover, by applying the maximum principle to the evolution equation of $|D\log u|_{g_{\mathbb{S}^2}}$, we have
\begin{lem}[\cite{Gerhardt2014}]
The smooth solution $u$ of \eqref{flow-u} satisfies the following $C^1$-estimate
\begin{equation}\label{C1-est}
  |D\log u(t,\theta)|_{g_{\mathbb{S}^2}}\leq ~|D\log u(0,\theta)|_{g_{\mathbb{S}^2}},\quad \forall~t\in[0,T),~\forall~\theta\in\mathbb{S}^2.
\end{equation}
\end{lem}

\subsection{Long time existence}

\begin{prop}\label{prop4.5}
Under the assumption of Theorem \ref{thm1.1}, the flow \eqref{1.1} exists for all time, i.e., $T=+\infty$.
\end{prop}
\proof
Firstly, applying the maximum principle to the evolution equation \eqref{2.3} gives the uniform upper bound on $F$:
\begin{equation}\label{F-upp1}
  F(x,t)\leq \max_{M_0}F,\quad \forall~(x,t)\in M\times [0,T).
\end{equation}
Moreover, Gerhardt \cite{Gerhardt2014} proved that there exists a positive constant $c_3$ depending only on $\alpha$ and $M_0$ such that for $\rho_1<\bar{r}<\rho_2$, we have
\begin{equation}\label{F-lowerbd}
  \rho(t,\bar{r})F(x,t)\geq c_3>0,\quad \forall~(x,t)\in M\times [0,T).
\end{equation}
Then if $T<\infty$, the estimates \eqref{F-lowerbd} and \eqref{F-upp1} imply that $F$ is bounded from above and below by positive constants depending on $\alpha, M_0$ and $T$. The bounds on $F$ and \eqref{F-lambda} give us the uniform upper and lower positive bounds on the principal curvatures, which combined with the $C^0$ estimate \eqref{ut2} and $C^1$-estimate \eqref{C1-est} yield the $C^2$-estimate of $u$. Moreover, by \eqref{pt-F-bd1} and the bounds on $F$,  the flow \eqref{1.1} and Eq. \eqref{flow-u} are uniformly parabolic. We may apply the second order derivative H\"{o}lder estimates in \cite{andrews2004fully} and parabolic Schauder estimates \cite{lieberman1996second} to derive uniform bounds on all higher derivatives of the principal curvatures and of $u(t)$. Then we have a smooth limit function $u_T:=\lim_{t\rightarrow T}u(t)$ which defines a smooth strictly convex surface $M_T$. The short time existence theorem then implies that we can continue the flow beyond the time $T$, which contradicts the definition of $T$. So we conclude that $T=\infty$.
\endproof

\begin{rem}
We note that the concavity of $F$ was not used in the proof of $C^0$ estimate \eqref{ut2}, $C^1$-estimate \eqref{C1-est}  and \eqref{F-lowerbd} in \cite{Gerhardt2014}, but it was used in the proof of the curvature estimate (see Lemma 4.10 in \cite{Gerhardt2014}) and applying the second derivative H\"{o}lder estimates of Krylov.
In order to overcome the difficulties without assumption of concavity, we use the pinching ratio estimate to derive curvature estimate.
In general, for the application of the second derivative H\"{o}lder estimates of Krylov \cite{Krylov}, we need that $F$ is concave in its arguments. In the two-dimensional case, Andrews \cite{andrews2004fully} proved that the second derivative H\"{o}lder estimates also hold without any concavity assumption on $F$.
\end{rem}

\subsection{Convergence}
First, we will use the curvature pinching estimate to refine estimate \eqref{F-upp1} on the upper bound of $F$.
\begin{lem}
Under the assumption of Theorem \ref{thm1.1},
there exists a positive constant $c_4$ depending only on $\alpha$ and $M_0$ such that for $\rho_1<\bar{r}<\rho_2$, we have
\begin{equation}\label{F-upp2}
  \rho(t,\bar{r})F\leq c_4,\quad \forall~t\in [0,\infty).
\end{equation}
\end{lem}
\proof
We consider the case $\alpha\in (0,1)$ and the case $\alpha=1$ separately.

For $\alpha\in (0,1)$, we define $\phi=\Phi+\mu \rho^{\alpha}(t,\bar{r})=-F^{-\alpha}+\mu \rho^{\alpha}(t,\bar{r})$, where $\mu \in (0,1]$ is a small constant such that $\phi<0$ on $M_0$ and
\begin{equation*}
  \mu<2^{-\alpha}\beta^{\frac{\alpha}{\alpha-1}}.
\end{equation*}
We aim to show that $\phi$ stays negative for such $\mu$ along the flow \eqref{1.1}. By \eqref{2.3} and \eqref{4.2}, we have
\begin{align}\label{F-upp2-pf1}
  \partial_t\phi-\dot{\Phi}\dot{F}^{ij}\nabla_i\nabla_j\phi =& \Phi\dot{\Phi}\dot{F}^{ij}h_{i}^kh_{kj}+\mu\alpha2^{-\alpha}\rho^{2\alpha-1}(t,\bar{r})\nonumber\\
  \leq &-\frac{\alpha}{2\beta}F^{-2\alpha+1}+\mu\alpha2^{-\alpha}\rho^{2\alpha-1}(t,\bar{r}),
\end{align}
where we used the estimate \eqref{3.12}. Let $t_0>0$ be the first time such that $\phi$ touches zero at some point $x_0\in M_{t_0}$. Then at this point, we have that
\begin{equation}\label{F-upp2-pf2}
  0=\phi(x_0,t_0)=-F^{-\alpha}(x_0,t_0)+\mu \rho^{\alpha}(t_0,\bar{r}).
\end{equation}
Applying the maximum principle to \eqref{F-upp2-pf1} and using \eqref{F-upp2-pf2}, we derive that
\begin{align*}
  0\leq & \mu\alpha 2^{-\alpha}\rho(t_0,\bar{r})^{2\alpha-1}\left(1-2^{\alpha-1}{\beta}^{-1}\mu^{1-\frac 1{\alpha}}\right)<0,
\end{align*}
which is a contradiction. Thus $\phi$ stays negative for all time $t\in [0,\infty)$ and we obtain that
\begin{equation*}
  \rho(t,\bar{r})F\leq \mu^{-\frac 1{\alpha}}.
\end{equation*}

For $\alpha=1$, we define $\phi=\Phi+\mu \chi$, where $\chi$ is the support function and $\mu$ is a positive constant such that $\phi\leq 0$ on $M_0$ . Combining \eqref{2.3} and \eqref{s2:evl-chi}, we have
\begin{equation}\label{F-upp2-pf3}
  \partial_t\phi-\dot{\Phi}\dot{F}^{ij}\nabla_i\nabla_j\phi=~\phi\dot{\Phi}\dot{F}^{ij}h_{i}^kh_{kj}.
\end{equation}
Since $\dot{\Phi}=F^{-2}$,  the estimates \eqref{3.12}-\eqref{3.13} give that $ (2\beta)^{-1}\leq\dot{\Phi} \dot{F}^{ij}h_{i}^kh_{kj}\leq {\beta}/2$. Then the coefficient of $\phi$ on the right of \eqref{F-upp2-pf3} is a uniformly bounded function.  By applying the maximum principle to \eqref{F-upp2-pf3}, we conclude that $\phi\leq 0$ for all $M_t$, $\forall~t\in [0,\infty)$. Therefore, we have $F\chi\leq 1/{\mu}$. On the other hand, by the $C^0, C^1$ estimates \eqref{ut2}, \eqref{C1-est}, we know that $\chi\geq c\rho$ for some positive constant $c$. Thus $F\rho\leq 1/{(c\mu)}$.
\endproof

Now we complete the proof of Theorem \ref{thm1.1}. Note that we can always find a constant $\bar{r}\in(\rho_1,\rho_2)$ such that
\begin{equation}\label{s4:u-td1}
  \lim_{t\to \infty}u(t,\theta)\rho^{-1}(t,\bar{r})=1.
\end{equation}
The estimate \eqref{s4:u-td1} follows from the Alexandrov reflection argument as in \cite[Lemma 3.5]{Gerhardt2014} by using the result of \cite{CG1996}.
If $\alpha\in(0,1)$, we present a simple proof for \eqref{s4:u-td1}. In this case, for any $\bar{r}\in(\rho_1,\rho_2)$, \eqref{s4:u-td1} follows from \eqref{ut} and
$$1=\lim_{t\to \infty}\rho(t,\rho_1)\rho^{-1}(t,\bar{r})\leq\lim_{t\to \infty}u(t,\theta)\rho^{-1}(t,\bar{r})\leq\lim_{t\to \infty}\rho(t,\rho_2)\rho^{-1}(t,\bar{r})=1,$$ so we get that  $\lim_{t\to \infty}u(t,\theta)\rho^{-1}(t,\bar{r})=1,~\forall ~\bar{r}\in(\rho_1,\rho_2).$

We now rescale the surface by
\begin{equation}\label{resca1}
  \tilde{X}(x,t)=\rho(t,\bar{r})^{-1}X(x,t).
\end{equation}
Define a new time function $\tau=\tau(t)$ by
\begin{equation}\label{rescal-t}
  \frac {d\tau}{dt}=\rho(t,\bar{r})^{\alpha-1}
\end{equation}
such that $\tau(0)=0$. Then $\tau$ ranges from $0$ to $\infty$. It's easy to check that the rescaled surface satisfies the following evolution equation
\begin{equation}\label{rescal-flow}
  \frac{\partial \tilde{X}}{\partial \tau}=\tilde{F}^{-\alpha}\nu-2^{-\alpha}\tilde{X}.
\end{equation}
Note that $ {\partial \tilde{F}}/{\partial \tilde{\lambda_i}}$ is homogeneous of degree zero and is scaling invariant, then from \eqref{pt-F-bd1} it is also bounded  on the rescaled surfaces. By the pinching estimate \eqref{3.10} and the bounds \eqref{F-lowerbd}, \eqref{F-upp2} on the rescaled speed function $\tilde{F}=\rho(t,\bar{r})F$, we have that the principal curvatures of the rescaled surfaces are also uniformly bounded from above and below by positive constants. Moreover, the uniform bounds on $\tilde{F}$ and $ {\partial \tilde{F}}/{\partial \tilde{\lambda_i}}$ imply that the flow \eqref{rescal-flow} is uniformly parabolic. By the standard argument using the H\"{o}lder estimates \cite{andrews2004fully}, parabolic Schauder estimates \cite{lieberman1996second} and interpolation inequalities, we can conclude that the rescaled flow  converges in $C^{\infty}$-topology to the unit sphere $\mathbb{S}^2$.

Finally, we show that the rescaled flow converges exponentially. Since the quantity $G$ defined in \eqref{G-def} is homogeneous of degree zero, on the rescaled surface $\tilde{M}_{\tau}$ we have $\tilde{G}=G$. Then
\begin{equation*}
  \frac{\partial}{\partial \tau}\tilde{G}=\rho(t,\bar{r})^{1-\alpha}\frac{\partial}{\partial t}G.
\end{equation*}
By the computation in \S \ref{sec:3},
\begin{align}\label{G-exp}
  \frac d{d\tau}\max_{\tilde{M}_{\tau}}\tilde{G} \leq & \left(Q_0+Q_1\right)\rho(t,\bar{r})^{1-\alpha} \nonumber\\
  \leq  &-4(1+\alpha)\frac{F^{-\alpha}\lambda_1\lambda_2}{\lambda_1+\lambda_2}\rho(t,\bar{r})^{1-\alpha} \max_{\tilde{M}_{\tau}}\tilde{G} \nonumber\\
  =&-4(1+\alpha)\frac{\tilde{F}^{-\alpha}\tilde{\lambda}_1\tilde{\lambda}_2}{\tilde{\lambda}_1+\tilde{\lambda}_2} \max_{\tilde{M}_{\tau}}\tilde{G} ~  \leq ~ -\delta \max_{\tilde{M}_{\tau}}\tilde{G},
\end{align}
where $\delta>0$ is a positive constant and in the last inequality we used the facts that $\tilde{\lambda}_i$ and $\tilde{F}$ are uniformly bounded from above and below by positive constants. \eqref{G-exp} implies that the trace-less part of the second fundamental form of $\tilde{M}_{\tau}$ has the following exponential decay
\begin{equation}
  |\mathring{\tilde{A}}|^2(x,\tau)\leq Ce^{-\delta\tau},\quad \forall~x\in\tilde{M}_{\tau},
\end{equation}
where $C=C(M_0)$. This gives by interpolation that
\begin{equation}
  |\nabla\mathring{\tilde{A}}|^2(x,\tau)\leq Ce^{-\delta\tau},\quad \forall~x\in\tilde{M}_{\tau},
\end{equation}
as we already have uniform bounds on $\nabla^k\mathring{\tilde{A}},~\forall~k\geq 0$. Note that in $2$-dimensional case, we have
\begin{equation*}
  |\nabla \tilde{A}|^2=|\nabla\mathring{\tilde{A}}|^2+\frac 12|\nabla \tilde{H}|^2\leq |\nabla\mathring{\tilde{A}}|^2+\frac 23|\nabla \tilde{A}|^2,
\end{equation*}
where we used the inequality (cf. \cite[\S 2]{huisken-1984})
\begin{equation*}
  |\nabla \tilde{H}|^2\leq \frac 43|\nabla \tilde{A}|^2.
\end{equation*}
Therefore, we obtain
\begin{equation}
  |\nabla \tilde{A}|^2(x,\tau)\leq ~3|\nabla\mathring{\tilde{A}}|^2(x,\tau)\leq ~Ce^{-\delta\tau},\quad \forall~x\in\tilde{M}_{\tau}
\end{equation}
and for all higher derivatives of $\tilde{A}$ by interpolation. Then the estimate on the metric and the exponential convergence of the immersions are the same as in \cite{Andrews1994-2}. This finishes the proof of Theorem \ref{thm1.1}.

\section{Flow in Hyperbolic space}
In this section,  we consider flow \eqref{1.1} in hyperbolic space $\mathbb{H}^3$ and prove Theorem \ref{thm1.2}. We assume that $F$ satisfies  Assumption \ref{assum-1} and $\alpha\in(0,1]$. Since the flow \eqref{1.1} is a parabolic equation with strictly convex initial data, by short time existence theorem we have a smooth solution on a maximal time interval $[0,T),~0<T\leq\infty.$ It remains to study the long time behavior of the flow.

Fix a point $x_0\in\mathbb{H}^3$. We consider geodesic polar coordinates centered at $x_0$. The metric on $\mathbb{H}^3$ can be expressed as
\begin{equation*}
  \bar{g}=dr^2+\sinh^2 rg_{\mathbb{S}^2}.
\end{equation*}
As in Euclidean case, if the initial surface is a geodesic sphere $S_{\rho_0}$ in $\mathbb{H}^3$, then along the flow \eqref{1.1}, the flow surfaces are also geodesic spheres with radius $\rho(t,\rho_0)$ solving the following ODE
\begin{equation}\label{ODE-H}
  \frac d{dt}\rho(t,\rho_0)=2^{-\alpha}\tanh^{\alpha}\rho(t,\rho_0).
\end{equation}
with $\rho(0,\rho_0)=\rho_0$. Since $2^{-\alpha}\tanh^{\alpha}\rho(t,\rho_0)\leq 2^{-\alpha}$, we have that \eqref{ODE-H} has solution for all $t\in[0,\infty)$. Moreover, it follows from $\tanh^{\alpha}\rho_0\leq \tanh^{\alpha}\rho(t,\rho_0)\leq 1$  that
\begin{equation}\label{rho-t-H}
  \rho_0+\frac t{2^{\alpha}}\tanh^{\alpha}\rho_0\leq \rho(t,\rho_0)\leq \rho_0+\frac t{2^{\alpha}}.
\end{equation}

Suppose that $M_0$ is a closed strictly convex surface in $\mathbb{H}^3$, then $M_0$ can be given by a graph of a positive function $u_0(\theta)$ over the geodesic sphere ${S}^2$ centered at some point $x_0$ in the enclosed domain by $M_0$. Under the assumption of Theorem \ref{thm1.2}, the flow \eqref{1.1} preserves the star-shaped condition (see Section \ref{sec:2-3}) and each flow surface $M_t$ can be written as a graph of a function $u(t,\theta)$ over the sphere. By \eqref{s2:ut-evl}, the defining function $u(t,\theta)$ satisfies the following scalar flow equation
\begin{equation}\label{flow-ut-H}
  \frac{\partial u}{\partial t}=vF^{-\alpha}
\end{equation}
with
\begin{equation}\label{v-def}
  v=\sqrt{1+|Du|^2_{g_{\mathbb{S}^2}}/{\sinh^{2}u}}.
\end{equation}
By applying the maximum principle, we have
\begin{lem}[\cite{gerhardt2011-H^n,Scheuer2015-1}]\label{lem-5-1}
\begin{itemize}
  \item[(1)] The solution $u$ of \eqref{flow-ut-H} satisfies that
\begin{equation}\label{ut-H-C0}
  \rho(t,\inf u(0,\cdot))\leq u(t,\theta)\leq \rho(t,\sup u(0,\cdot)), \quad \forall~t\in[0,T),~\forall~\theta\in\mathbb{S}^2.
\end{equation}
  \item[(2)] There exists a constant $c$ depending on $M_0$ and $\alpha$ such that
  \begin{equation*}
    0<c^{-1}\leq e^{-\frac t{2^{\alpha}}}\sinh u(t,\theta)\leq c,\quad \forall~t\in[0,T),~\forall~\theta\in\mathbb{S}^2,
  \end{equation*}
  and
  \begin{equation*}
    \coth u(t,\theta)-1\leq ce^{-2^{(1-\alpha)}\cdot t},~\forall~t\in[0,T),~\forall~\theta\in\mathbb{S}^2.
  \end{equation*}
\end{itemize}
\end{lem}

\begin{lem}[\cite{gerhardt2011-H^n,Scheuer2015-1}]\label{lem-5-2}
The solution $u$ of \eqref{flow-ut-H} satisfies the following $C^1$-estimate
\begin{equation}
  v(t,\theta)\leq~\sup v(0,\cdot),\quad \forall~t\in[0,T),~\forall~\theta\in\mathbb{S}^2,
\end{equation}
where $v$ is defined in \eqref{v-def}.
\end{lem}

\begin{rem}
We note that the concavity of $F$ was not used in the proof of $C^0$ estimate and $C^1$-estimate in \cite{gerhardt2011-H^n,Scheuer2015-1}, it was used in the proof of the curvature estimate (see Lemma 4.4 in \cite{gerhardt2011-H^n} and Proposition 3.11 in \cite{Scheuer2015-1}) and applying the second derivative H\"{o}lder estimates of Krylov.
\end{rem}

\begin{lem}
Under the assumption of Theorem \ref{thm1.2}, there exists a constant $c$ depending only on $M_0,\alpha$ such that
\begin{equation}\label{F-bdd-twoside}
  0<c^{-1}\leq F(x,t)\leq c,\quad \forall~(x,t)\in M\times [0,T).
\end{equation}
\end{lem}
\proof
The proof is essentially given in \cite{gerhardt2011-H^n,Scheuer2015-1}. The only difference is that when deriving the lower bound of $F$, the inequality $\dot{F}^{ij}g_{ij}\geq 2$ which is due to the concavity of $F$, is used crucially in \cite{gerhardt2011-H^n,Scheuer2015-1}. In our case, the pinching estimate \eqref{3.10} implies a similar lower bound $ \dot{F}^{ij}g_{ij}\geq 2/{\beta}$ given in \eqref{F-bd-H-pf2}. This is enough to derive the estimate \eqref{F-bdd-twoside}.  \endproof

The estimate \eqref{F-bdd-twoside} on $F$ and \eqref{pt-F-bd1} on $\partial F/{\partial\lambda_i}, i=1,2,$ give us the uniform upper and lower positive bounds on the principal curvatures $\lambda_1,\lambda_2$. Using the same argument as that in the proof of Proposition \ref{prop4.5}, we derive that
\begin{prop}
Under the assumption of Theorem \ref{thm1.2}, the flow \eqref{1.1} in $\mathbb{H}^3$ exists for all time, i.e., $T=\infty$.
\end{prop}

Up to now, we have proved that the curvature function $F$ and the principal curvatures $\lambda_1,\lambda_2$ of $M_t$ are bounded from above and below by uniform positive constants depending only on $\alpha$ and $M_0$ for all time $t\in [0,\infty)$.

\begin{lem}\label{lem-pinch-H}
Under the assumption of Theorem \ref{thm1.2}, the pinching ratio $\beta$ of $M_t$ tends to $1$ exponentially as $t\to\infty$.
\end{lem}
\proof
From the proof of Theorem \ref{thm3.1}, we know that
\begin{align*}
  \frac{d}{dt}\max_{M_t}G \leq & -\frac{4F^{-\alpha}}{\lambda_1+\lambda_2}\left(1-\alpha+(1+\alpha)\lambda_1\lambda_2\right)\max_{M_t}G ~\leq  ~-C\max_{M_t}G,
\end{align*}
where we used the estimates on $F$ and $\lambda_1,\lambda_2$ and therefore $C>0$ is a positive constant depending only on $\alpha$ and $M_0$.  Then
\begin{equation*}
  G\leq e^{-Ct}\max_{M_0}G\to 0,\quad \text{as}~~t\to \infty.
\end{equation*}
\endproof

From Section \ref{sec:2-2}, we know that for a graph $M=\text{graph} ~u$ over a geodesic sphere in $\mathbb{H}^3$,  the induced metric on $M$ is
\begin{equation}\label{g-H}
  g_{ij}=u_iu_j+\sinh^2 u \sigma_{ij},
\end{equation}
where $\sigma_{ij}$ is the components of $g_{\mathbb{S}^2}$. The second fundamental form of $M$ satisfies
\begin{equation}\label{2ff-H}
  h_{ij}v^{-1}=-u_{ij}+\sinh u\cosh u\sigma_{ij},
\end{equation}
where $v$ is defined in \eqref{v-def} and the covariant derivatives $u_{ij}$ are taken with respect to the induced metric of $M$. In Theorem 4.1 of \cite{Scheuer2015-1}, it was proved that there exist constants $\lambda>0$, $c>0$ depending on $\alpha,M_0$ such that
\begin{equation}\label{v-decay}
  v-1\leq~ce^{-\lambda t},\quad \forall ~t\in [0,\infty).
\end{equation}
Note that the concavity of $F$ is not required in the proof of \eqref{v-decay}. Recall that the support function $\chi$ of the graph $M=\text{graph}~  u$ in $\mathbb{H}^3$ is defined by
\begin{equation*}
  \chi=\bar{g}((\sinh u)\partial_r,\nu)=v^{-1}\sinh u.
\end{equation*}
In view of the estimate \eqref{v-decay} and Lemma \ref{lem-5-1}, along the flow \eqref{1.1} in $\mathbb{H}^3$, we have that
\begin{equation}\label{lim-chi}
  \sup_{M_t}\left |\log\chi-u+\log 2\right |=~ \sup_{M_t}|\log(2e^{-u}v^{-1}\sinh u)|\leq ~ce^{-\lambda t}.
\end{equation}
By \eqref{s2:chi-ds1}, the first order derivative of the support function $\chi$ satisfies $\chi_i=\bar{g}( \sinh u~ \partial_r,h_i^k\partial_{k})$, so we have
\begin{equation}\label{chi-bd}
  \chi^{-1}|\chi_i|\leq \chi^{-1}(\sum_{i=1}^2\chi_i^2)^{\frac 12}\leq \lambda_2\frac{(1-\bar{g}(\partial_r,\nu)^2)^{1/2}}{\bar{g}(\partial_r,\nu)}=~ \lambda_2\sqrt{v^2-1}\leq ~ce^{-\lambda t}
\end{equation}
on $M_t$, as the principal curvature $\lambda_2$ is uniformly bounded. We note that the $c$ and $\lambda$ in Eqs. \eqref{v-decay}, \eqref{lim-chi} and \eqref{chi-bd} are not the same, but all of them only depend on $\alpha$ and $M_0$.

\begin{lem}\label{lem-cur1}
Under the assumption of Theorem \ref{thm1.2}, the principal curvatures $\lambda_1,\lambda_2$ of $M_t$ converge to $1$ as $t\to\infty$.
\end{lem}
\proof
Recall that we have the following relation between $F$ and $\lambda_i$ (see \eqref{F-lambda}) which is due to the pinching estimate:
\begin{equation*}
  2\lambda_1\leq F(\lambda_1,\lambda_2)\leq 2\lambda_2\leq2\beta\lambda_1.
\end{equation*}
Since the pinching ratio $\beta$ of $M_t$ tends to $1$ as $t\to \infty$ by Lemma \ref{lem-pinch-H}, in order to prove this lemma it suffices to prove
\begin{equation}\label{lem-cur1-pf1}
  \sup_{M_t}|F(x,t)-2|\to 0,\quad \text{as}~~t\to \infty.
\end{equation}

(i). First,  we prove
\begin{equation}\label{lem-cur1-pf2}
  \liminf_{t\to\infty}\inf_{M_t}F(x,t)\geq 2.
\end{equation}
The proof is similar to part (ii) of Lemma 4.3 in \cite{Scheuer2015-1}. We define a function on $M_t$ by
\begin{equation}\label{omeg-H}
  \omega=\log{(-\Phi)}-\log\chi+u+(\alpha-1)\log 2,
\end{equation}
which is slightly different from the one in part (ii) of Lemma 4.3 of \cite{Scheuer2015-1}. In view of the estimate \eqref{lim-chi}, in order to prove \eqref{lem-cur1-pf2}, it suffices to prove that
\begin{equation}\label{td-omega-lim}
  \limsup_{t\to\infty}\sup_{M_t}\omega\leq 0.
\end{equation}
By using \eqref{2.3}, \eqref{flow-ut-H} and \eqref{s2:evl-chi}, we obtain  that
\begin{align}\label{omeg-evl1}
  \partial_t\omega= & \Phi^{-1}\partial_t\Phi-\chi^{-1}\partial_t\chi+\partial_tu\nonumber\\
  =&\dot{\Phi}\dot{F}^{ij}(\Phi^{-1}\nabla_i\nabla_j\Phi-\chi^{-1}\nabla_i\nabla_j\chi)-\dot{\Phi}\dot{F}^{ij}g_{ij}+(1-\alpha){\Phi}{\chi}^{-1}\cosh u+vF^{-\alpha}\nonumber\\
   =&\dot{\Phi}\dot{F}^{ij}\omega_{ij}+ \dot{\Phi}\dot{F}^{ij}\left(\omega_i\omega_j+2\omega_i(\chi^{-1}\chi_j-u_j)-2\chi^{-1}\chi_iu_j+u_iu_j-u_{ij}\right)\nonumber\\
  &\quad -\dot{\Phi}\dot{F}^{ij}g_{ij}+(1-\alpha){\Phi}{\chi}^{-1}\cosh u+vF^{-\alpha}\nonumber\\
  = & \dot{\Phi}\dot{F}^{ij}\omega_{ij}+ \dot{\Phi}\dot{F}^{ij}\left(\omega_i\omega_j+2\omega_i(\chi^{-1}\chi_j-u_j)-2\chi^{-1}\chi_iu_j+(1+\coth u)u_iu_j\right)\nonumber\\
  &\quad -(1+\coth u)\dot{\Phi}\dot{F}^{ij}g_{ij}+\dot{\Phi}\dot{F}^{ij}h_{ij}v^{-1}+(1-\alpha){\Phi}{\chi}^{-1}\cosh u+vF^{-\alpha},
\end{align}
where in the last equality we used the expressions \eqref{g-H} and \eqref{2ff-H}. Note that all the covariant derivatives in the equation above  are taken with respect to the induced metric on $M_t$. Denote $\tilde{\omega}(t)=\sup_{M_t}\omega=\omega(t,\theta_t)$, which is a Lipschitz function on $\mathbb{R}^+$. Then at the point $(t,\theta_t)$, $\dot{\Phi}\dot{F}^{ij}\omega_{ij}\leq 0$ and $\omega_i=0$. Moreover, by using the estimates \eqref{v-decay}, \eqref{chi-bd} and noting that $|\nabla u(t)|^2_{g_t}=\sinh^{-2}u(t)v(t)^{-2}|Du(t)|^2_{g_{\mathbb{S}^2}}$, we have that
\begin{equation*}
  \dot{\Phi}\dot{F}^{ij}\left(-2\chi^{-1}\chi_iu_j+(1+\coth u)u_iu_j\right)~\to~0, \quad \text{as }t\to\infty.
\end{equation*}
Then from \eqref{omeg-evl1}, we have that
\begin{align}\label{evl-omeg-H}
  \frac d{dt}\tilde{\omega}(t)\leq & o(1)-4\alpha\beta^{-1}F^{-\alpha-1}(t,\theta_t)+(2\alpha-1+v)F^{-\alpha}(t,\theta_t)\nonumber\\
  \leq&~o(1)+2\alpha(F(t,\theta_t)-2)F^{-\alpha-1}(t,\theta_t)\quad \text{as }t\to\infty,
\end{align}
where we used that $\dot{F}^{ij}g_{ij}\geq  2/\beta$~(see \eqref{F-bd-H-pf2}), $\coth u\geq 1$, and the estimates $\beta,v\to 1$ as $t\to \infty$ which are obtained in \eqref{v-decay} and Lemma \ref{lem-pinch-H}.

We claim that: $\forall~\epsilon>0$, there exist $t_\epsilon>0$ and $\delta_\epsilon>0$ such that
\begin{equation}\label{claim-H}
 A_{\epsilon}:= \big\{t\in [t_\epsilon,\infty)\cap D:\tilde{\omega}(t)>\epsilon\big\}~\subset~\big\{t\in [t_\epsilon,\infty)\cap D: \frac d{dt}\tilde{\omega}(t)\leq -\delta_\epsilon\big\},
\end{equation}
where $D$ is the set of points of differentiability of $\tilde{\omega}$.

Since we have already proved that $F$ is bounded, there exists a constant $c$ which depends on $\alpha$ and $M_0$ such that
 $2\alpha F^{-\alpha-1}\geq c(\alpha,M_0)$. For any $\epsilon>0$, we define $\delta_\epsilon=c(1-e^{-\frac{\epsilon}{2\alpha}})>0$, which depends on $\alpha, M_0$ and $\epsilon$. By the estimate \eqref{lim-chi} and \eqref{evl-omeg-H}, we can choose $t_\epsilon>0$ sufficiently large such that
\begin{equation}\label{F-inf}
  -\log\chi+u-\log 2~<\frac {\epsilon}2,~ \frac d{dt}\tilde{\omega}(t)\leq~ \delta_\epsilon+2\alpha(F(t,\theta_t)-2)F^{-\alpha-1}(t,\theta_t),~\forall ~t\in[t_\epsilon,\infty).
\end{equation}
Then for any $t\in A_{\epsilon}$, $\log(-\Phi 2^{\alpha})(t,\theta_t)>\epsilon/2$, which implies that $F(t,\theta_t)<2e^{-\frac{\epsilon}{2\alpha}}$. Then in view of \eqref{F-inf}, we have
\begin{equation*}
  \frac d{dt}\tilde{\omega}(t)\leq~ \delta_\epsilon+c(2e^{-\frac{\epsilon}{2\alpha}}-2)=-\delta_\epsilon.
\end{equation*}
This proves the claim \eqref{claim-H}. By an easy exercise of Calculus (cf. Lemma 4.2 in \cite{Scheuer2015-1}), the estimate \eqref{td-omega-lim} follows from the claim \eqref{claim-H}.

(ii). Second, we prove the other direction
\begin{equation}\label{lem-cur1-pf3}
  \limsup_{t\to\infty}\sup_{M_t}F(x,t)\leq 2.
\end{equation}
We modify \eqref{omeg-H} and define a new function on $M_t$  by
\begin{equation*}
  \phi=-\log(-\Phi)-\log\chi+u-(\alpha+1)\log 2.
\end{equation*}
We aim to prove that
\begin{equation}
  \limsup_{t\to\infty}\sup_{M_t}\phi\leq 0,
\end{equation}
which is equivalent to \eqref{lem-cur1-pf3} in view of the estimate \eqref{lim-chi}. By a direct calculation, $\phi$ satisfies the following evolution equation
\begin{align*}
  \partial_t\phi= & \dot{\Phi}\dot{F}^{ij}\phi_{ij}- \dot{\Phi}\dot{F}^{ij}\left(\phi_i\phi_j+2\phi_i(\chi^{-1}\chi_j-u_j)+2\chi^{-2}\chi_i\chi_j-2\chi^{-1}\chi_iu_j+(1-\coth u)u_iu_j\right)\nonumber\\
  &\quad -2\dot{\Phi}\dot{F}^{ij}h_i^kh_{kj}+(1-\coth u)\dot{\Phi}\dot{F}^{ij}g_{ij} +\dot{\Phi}\dot{F}^{ij}h_{ij}v^{-1}+(1-\alpha){\Phi}{\chi}^{-1}\cosh u+vF^{-\alpha}.
\end{align*}
Denote $\tilde{\phi}(t)=\sup_{M_t}\phi=\phi(t,\theta_t)$, which is also a Lipschitz function on $\mathbb{R}^+$. Then
\begin{align*}
  \frac d{dt}\tilde{\phi}(t)\leq & o(1)-\alpha\beta^{-1}F^{-\alpha+1}(t,\theta_t)+\alpha( v^{-1}+v)F^{-\alpha}(t,\theta_t)\nonumber\\
  \leq&~o(1)+\alpha(2-F(t,\theta_t))F^{-\alpha}(t,\theta_t) \quad \text{as }t\to\infty,
\end{align*}
where we used that $\dot{F}^{ij}h_i^kh_{kj}\geq F^2/{2\beta}$~(see \eqref{3.12}), $\coth u\geq 1$ and $\beta,v\to 1$ as $t\to\infty$~(see \eqref{v-decay} and Lemma \ref{lem-pinch-H}). Then the remaining proof follows after a similar argument to that in part (i).
\endproof

Now we can follow the similar argument as that in \cite{Scheuer2015-1} to complete the proof of Theorem \ref{thm1.2}. First, by the same argument as that in Theorem 4.4 of \cite{Scheuer2015-1}, we can obtain the optimal convergence rate of the principal curvatures, i.e.,
\begin{equation}\label{decay-cur}
  |h_i^j-\delta_i^j|~\leq~ce^{-2^{(1-\alpha)}\cdot t},\quad \forall~t>0,
\end{equation}
where $c$ is a positive constant depending on $\alpha$ and $M_0$. The conclusion of Lemma \ref{lem-cur1} plays a crucial role in dealing with the bad terms involving the derivatives of the second fundamental form. Second, using the argument in \cite[\S 5]{Scheuer2015-1}, we can also obtain uniform higher order derivative estimates. Finally, using the conformally flat parametrization as that in \cite[\S 6]{Scheuer2015-1}, we obtain the convergence \eqref{thm1.2-f} of the defining function $u(t,\theta)$. We refer the reader to \cite{Scheuer2015-1} for more details.

\section{Flow in sphere}

In this section,  we consider the flow \eqref{1.1} in sphere $\mathbb{S}^3$ and prove Theorem \ref{thm1.3}. We assume that $F$ satisfies  Assumption \ref{assum-1} and $\alpha=1$. Since $M_0$ is assumed to be strictly convex in $\mathbb{S}^3$, then $M_0$ is embedded, contained in an open hemisphere and is the boundary of a convex body in $\mathbb{S}^3$.
Since the flow  is a parabolic equation with strictly convex initial data, by short time existence theorem, the flow surfaces $M_t$ exist on a maximal time interval $[0,T)$ for some $0<T\leq \infty$.

Firstly, we show that $T$ is finite and is characterized as the time when the velocity of the flow blows up. Recall the evolution equation \eqref{2.3} of $\Phi$,
\begin{equation}\label{6.1}
  \frac{\partial}{\partial t}\Phi-\dot{\Phi}\dot{F}^{ij}\Phi_{ij}=\Phi\dot{\Phi}(\dot{F}^{ij}h_{ip}h^p_j+\dot{F}^{ij}g_{ij}).
\end{equation}
As before, the pinching ratio estimate \eqref{3.10} implies that (cf. \eqref{F-bd-H-pf2} and \eqref{3.12})
\begin{equation}\label{6.2}
  \dot{F}^{ij}h_{ip}h^p_j+\dot{F}^{ij}g_{ij}\geq \frac 1{2\beta}F^{2}+\frac 2{\beta}.
\end{equation}
Recall that $\Phi=-F^{-1}$ in this case, then we have
\begin{equation}\label{prop6-1-pf1}
  \frac{\partial}{\partial t}(F^{-1})-\dot{\Phi}\dot{F}^{ij}(F^{-1})_{ij}\geq \frac 2{\beta}F^{-3}.
\end{equation}
The maximum principle then implies that $F^{-1}$ will blow up in finite time and thus $T<\infty$.

From \eqref{prop6-1-pf1}, we can also see that $F^{-1}$ is bounded from below by its initial value, and therefore $F$ is uniformly bounded above. By the pinching ratio estimate, we conclude that the principal curvatures $\lambda_1,\lambda_2$ are also uniformly bounded above. Then as in previous two sections, we can derive $C^0, C^1$ and $C^2$ estimates of the flow surfaces. Moreover, arguing as \cite[\S 5-6]{makowski2013rigidity} and using the pinching ratio estimate \eqref{3.10} and Andrews' \cite{andrews2004fully} second order derivative H\"{o}lder estimates and parabolic Schauder estimates \cite{lieberman1996second}, we can conclude that  $\liminf_{t\to T}\inf_{M_t}F\to 0$ as $t\to T$, $M_t\to M_T$ in $C^{1,\gamma}, 0<\gamma<1$ as $t\to T$ and that there exists a point $x_0\in\mathbb{S}^3$ such that the limit surface $M_T$ is the equator $S(x_0)=\partial \mathcal{H}(x_0)$. More precisely, we have

\begin{prop}\label{s6:prop-1}
Under the assumption of Theorem \ref{thm1.3}, the flow surfaces $M_t$ converge to the equator $S(x_0)=\partial \mathcal{H}(x_0)$ in $C^{1,\gamma}, 0<\gamma<1$ as $t\to T$.
\end{prop}

To show the smooth convergence, we employ the dual flow method introduced by Gerhardt in \cite{Gerhardt2015}. Let $M_0\subset \mathbb{S}^{n+1}$ be a closed strictly convex hypersurface given by the immersion
\begin{equation}\label{s2:embd-X}
  X:M\to M_0\subset \mathbb{S}^{n+1}.
\end{equation}
Then $M_0$ is embedded and homeomorphic to $\mathbb{S}^n$, contained in an open hemisphere and is the boundary of a convex body $\Omega_0\subset \mathbb{S}^{n+1}$. $M_0$ can also be viewed as a submanifold of codimension 2 in $\mathbb{R}^{n+2}$. The Gauss formula then is
\begin{equation*}
  X_{ij}~=-g_{ij}X-h_{ij}\tilde{X},
\end{equation*}
where $X_{ij}$ denotes the covariant derivatives of $X$, $g_{ij}$ is the induced metric, $h_{ij}$ the second fundamental form of $M_0$ considered as a hypersurface of $\mathbb{S}^{n+1}$, and $\tilde{X}$ represents the exterior normal vector of $M_0$ in $\mathbb{S}^{n+1}$. Then the map $ \tilde{X}:M\to \mathbb{S}^{n+1}$ is an embedding of a closed strictly convex hypersurface $\tilde{M}_0=\tilde{X}(M)$, which is called the Gauss map of $M_0$.

Viewing $\tilde{M}_0$ as a codimension 2 submanifold in $\mathbb{R}^{n+2}$, its Gauss formula is
\begin{equation*}
  \tilde{X}_{ij}~=-\tilde{g}_{ij}\tilde{X}-\tilde{h}_{ij}X,
\end{equation*}
where $X$ is the embedding \eqref{s2:embd-X} which represents the exterior normal vector of $\tilde{M}_0$. The second fundamental forms of $M_0, \tilde{M}_0$  and the corresponding principal curvatures $\lambda_i, \tilde{\lambda}_i$ satisfy $h_{ij}= \tilde{h}_{ij}=\langle \tilde{X}_i,X_j\rangle$ and $\tilde{\lambda}_i=\lambda_i^{-1}$, where $\langle\cdot,\cdot\rangle$ denotes the Euclidean metric of $\mathbb{R}^{n+2}$. We have the following result from \cite{Gerhardt2015} which does not need the concavity of $F$.

\begin{thm}[\cite{Gerhardt2015}]\label{thm6.2}
Let $\Phi\in C^{\infty}(\mathbb{R}_+)$ be strictly monotone and $F$ be the curvature function satisfying Assumption \ref{assum-1}. Assume that $M_0\subset \mathbb{S}^{n+1}$ is a closed and strictly convex  hypersurface and $\tilde{M}_0$ is its dual hypersurface. We have that the dual flows
\begin{equation}\label{s6:dual-flow1}
 \frac{\partial}{\partial t}X=~-\Phi(F)\nu
 \end{equation}
and \begin{equation}\label{s6:dual-flow2}
 \frac{\partial}{\partial t}\tilde{X}=~\Phi(\tilde{F}^{-1})\tilde{\nu}
 \end{equation}
 with initial data $M_0$ and $\tilde{M}_0$ resp. exist on maximal time intervals $[0,T)$ resp. $[0,T^*)$, where the flow hypersurfaces are strictly convex. Moreover, $T=T^*$ and the corresponding flow hypersurfaces $M_t$ and $\tilde{M}_t$ are polar sets of each other.
\end{thm}

Since $\Phi(F)=-F^{-1}$, we have $\Phi(\tilde{F}^{-1})=-\tilde{F}$. Then the flow \eqref{s6:dual-flow2} corresponds to the contracting curvature flow in sphere with speed $\tilde{F}$.  McCoy \cite{McCoy2017} proved the convergence result for such contracting curvature flow in sphere without needing concavity of the speed function: Let $X: M\times [0,T)\to \mathbb{S}^3$ be a solution to the contracting curvature flow
\begin{equation}\label{s6:flow-contr}
 \frac{\partial}{\partial t}X=~-F\nu,
 \end{equation}
with speed $F=F(\lambda_1,\lambda_2)$ satisfying Assumption \ref{assum-1}. Using the argument in \cite[\S 6--7]{Gerhardt2015} and the $C^{2,\alpha}$ estimate without needing concavity of the speed function by Andrews \cite{andrews2004fully}, McCoy \cite{McCoy2017} first proved that the solution ${M}_t$ of the flow \eqref{s6:flow-contr} contracts any strictly convex surface to a point $p\in \mathbb{S}^3$ as $t\to T$. Let $\Theta=\Theta(t,T)$ denote the radius of the shrinking spheres along the flow \eqref{s6:flow-contr} with extinction time $T$, i.e., $\Theta(t,T)=\arccos e^{2(t-T)}$ which satisfies
\begin{equation*}
  \frac d{dt}\Theta(t,T)=-2\cot\Theta(t,T),\quad \Theta(T,T)=0.
\end{equation*}
Then it has been proved in \cite{McCoy2017} that the rescaled principal curvatures satisfy
\begin{equation}\label{s6:kap}
0<C^{-1}\leq \lambda_i\Theta\leq C,\qquad i=1,2
\end{equation}
for some constant $C>0$ and the rescaled polar graph function $\tilde{u}(\cdot,\tau)=u/{\Theta}$ has uniform $C^m$ regularity for all $m\geq 2$, where $u$ is the polar graph function of the solution $M_t$ with respect to the point $p$ and the rescaled time parameter is $\tau=-\ln\Theta$. Note that the rescaled polar graph function $\tilde{u}$ satisfies
\begin{equation}\label{s6:u-td}
  \frac{\partial}{\partial \tau}\tilde{u}=-\frac 12vF\tan e^{-\tau} +\tilde{u},
\end{equation}
where $F=F(h_{i}^j)$ is the function evaluated at the Weingarten matrix of the unrescaled solution $M_t=\mathrm{graph} (u)$ and $v=\sqrt{1+\sin^{-2}u|Du|^2_{g_{\mathbb{S}^2}}}$ is given in \eqref{s2:v-def}. Moreover, $\tilde{u}(\cdot,\tau)$ converges to $1$ smoothly as $\tau\to\infty$. Since the detail is not given in \cite{McCoy2017} and the convergence will be used crucially in our dual expanding curvature flow,  for the readers' convenience we describe the convergence of the flow \eqref{s6:flow-contr} in the following by using the rescaling process in \cite[\S 5-6]{Andrews1994-3}.

For any constant $A>0$, we define $X^{(A)}: M\times [0,A^2T)\to (\mathbb{S}^3, A^2g_{\mathbb{S}^3})$ by $X^{(A)}(x,t)=X(x,t/{A^2})$. Then the principal curvatures $\lambda_i^{(A)}$ and the unit normal vector field on $M_t^{(A)}=X^{(A)}(M,t)$ satisfy $\lambda_i^{(A)}=\lambda_i/A$ and $\nu^{(A)}=\nu/A$. Therefore, $X^{(A)}$ is also a solution to the equation \eqref{s6:flow-contr} in $(\mathbb{S}^3, A^2g_{\mathbb{S}^3})$. We consider a sequence of times $\{t_k\}$ approaching to the maximal existence time $T$ of $X$, such that
 \begin{equation*}
   \sup_{M\times [0,t_k)}|\mathcal{W}|(x,t)=|\mathcal{W}|(x_k,t_k)
 \end{equation*}
for some point $x_k\in M_{t_k}$. For each $k$, we rescale the flow \eqref{s6:flow-contr} with respect to the parameter $A=A_k=|\mathcal{W}|(x_k,t_k)$ and define a sequence of flows $X^{k}(x,t)=X^{(A_k)}(x,t)=X(x,t_k+t/A_k^2)$, where $t\in [0, A_k^2(T-t_k))$. Without loss of generality, we can assume that $k$ is  large enough such that $t_k$ is sufficiently close to $T$. Since $\Theta(t_k,T)=\arccos e^{2(t_k-T)}$, we have $\Theta(t_k,T)\approx 2\sqrt{(T-t_k)}$ when $k$ is sufficiently large. \eqref{s6:kap} implies that $0<C^{-1}\leq A_k^2(T-t_k)\approx \lambda_i^2\Theta^2\leq C$ for another constant $C>0$. Therefore, there exists a constant $\delta>0$ such that $X^k(x,t)$ exists on the interval $t\in [0,\delta]$ for all $k$.

For each large $k$, we choose an isometry from $\mathbb{R}^3$ to $T_{x_k}\mathbb{S}^3$ so that we can identify the tangent spaces of $\mathbb{S}^3$ at each $x_k$. Using the exponential maps from $T_{x_k}\mathbb{S}^3$ to $\mathbb{S}^3$, we obtain a family of surfaces $\check{M}_t^k$ in $\mathbb{R}^3$ which corresponds to the family $M_t^k=X^k(M,t)\subset \mathbb{S}^3$. The regularity estimate of $M_t^k$ gives the corresponding regularity estimate of $\check{M}_t^k$.  From \eqref{s6:kap} we know that $u A_k\approx u/\Theta$, so we also have that $\check{M}_t^k$ are contained in a bounded region of $\mathbb{R}^3$.  Since $A_k=|\mathcal{W}|(x_k,t_k)$ goes to infinity as $k\to\infty$, the  metric on $\mathbb{R}^3$  induced by the exponential maps from $T_{x_k}\mathbb{S}^3$ to $(\mathbb{S}^3, A_k^2g_{\mathbb{S}^3})$ will converge to a flat metric as $k\to \infty$. We can find a subsequence $\{k'\}$ of $\{k\}$ such that  $\check{M}_t^{k'}$ converges to a limiting family of complete surfaces $\check{M}_t^{\infty}$ in $\mathbb{R}^3$ as $k'\to\infty$. Since each surface $\check{M}_t^{\infty}$ has pinched principal curvatures, a theorem by Hamilton \cite{Ham94} implies that $\check{M}_t^{\infty}$ is compact. Moreover, $\check{M}_t^{\infty}$ is a smooth solution of the flow \eqref{s6:flow-contr} in $\mathbb{R}^3$ for $t\in [0,\delta]$ with $\delta>0$.

The limit solution $\check{M}_t^{\infty}$ must be shrinking spheres in $\mathbb{R}^3$. This follows from the evolution equation of the following quantity
\begin{equation*}
  G~=~\frac{(\lambda_2-\lambda_1)^2}{(\lambda_1+\lambda_2)^2}.
\end{equation*}
In \cite{Andrews2010}, Andrews calculated that along the flow \eqref{s6:flow-contr} in $\mathbb{R}^3$, $G$ satisfies
\begin{equation}\label{s6:G-evl}
  \frac{\partial}{\partial t}G=\dot{F}^{kl}\nabla_k\nabla_lG+Q(\nabla h,\nabla h),
\end{equation}
where $Q(\nabla h,\nabla h)$ are quadratic terms involving the first derivatives of the second fundamental form. At the maximum point of $G$ where $G$ is non-zero (otherwise, $\check{M}_t^{\infty}$ is a sphere and the proof is trivial),  the gradient terms $Q(\nabla h,\nabla h)$ in \eqref{s6:G-evl} satisfy
\begin{equation}\label{s6:Q}
Q(\nabla h,\nabla h)=-\frac{8F}{H^3}\left((\nabla_1h_{22})^2+(\nabla_2h_{11})^2\right)
\end{equation}
and are non-positive. The strong maximum principle implies that the maximum of $G$ is strictly decreasing along the limit solution unless $G$ is a constant in both space and time. Since the quantity $G$ is a scaling-invariant, the maximum of $G$ can not be strictly decreasing on the limit solution. Hence, the value of $G$ on the limit solution $\check{M}_t^{\infty}$ must be a constant, independent of space and time. In fact, this constant corresponds to the limit of $\max_{M_{t_k}}G$ as $t_k\to T$ along the unrescaled flow \eqref{s6:flow-contr} in $\mathbb{S}^3$. Moreover, this constant must be zero: If this constant is not zero, then \eqref{s6:Q} implies that $\nabla_1h_{22}\equiv 0$ and $\nabla_{2}h_{11}\equiv 0$. The Codazzi equations imply that $\nabla_{1}h_{12}\equiv0$ and $\nabla_2h_{12}\equiv 0$. The fact $\nabla G\equiv 0$ then implies that $\nabla_{i}h_{jk}\equiv 0$ for all $i,j,k$. Therefore, $\check{M}_t^{\infty}$ are shrinking spheres. This implies that there exists a subsequence of time $t_k$ such that the maximum of $G$ on $M_{t_k}$ along the flow \eqref{s6:flow-contr} in $\mathbb{S}^3$ converges to zero as $t_k\to T$. Since the maximum of $G$ is monotone decreasing along the flow \eqref{s6:flow-contr}, we conclude that $\max_{M_t}G$  converges to zero as $t\to T$. That is, the solution $M_t$ of the flow \eqref{s6:flow-contr} converges to a round point as $t\to T$. The stronger exponentially smooth convergence of the rescaled equation \eqref{s6:u-td} can be deduced by considering the linearization of the flow about the shrinking sphere solution.

Now we go back to the expanding curvature flow \eqref{s6:dual-flow1}. As was shown in Proposition \ref{s6:prop-1}, the solution $M_t$ of the flow \eqref{1.1} will expand to the equator $\partial \mathcal{H}(x_0)$ as $t\to T$. There exists a time $\bar{t}<T$ such that $M_t$ can be written as a polar graph of a function $u$ with respect to the geodesic polar coordinates centered at $x_0$ for all $t\in [\bar{t},T)$. The dual contracting curvature flow can be written as graphs of a function $u^*$ with respect to the geodesic polar coordinates centered at $-x_0$ for all $t\in [\bar{t},T)$. By Lemma 9.1 of \cite{Gerhardt2015}, the functions $u, u^*$ satisfy the relations
\begin{equation*}
  u_{\max}=\frac{\pi}2-u^*_{\min},\qquad u_{\min}=\frac{\pi}2-u^*_{\max},\qquad \forall~t\in [\bar{t},T).
\end{equation*}
Since the rescaled function $u^*\Theta^{-1}\to 1$ as $t\to T$, we have $\omega=(\frac{\pi}2-u)\Theta^{-1}\to 1$ as $t\to T$. The $C^1$ estimate of $\omega$ follows from Lemma 9.3 of \cite{Gerhardt2015}. The dual relation and the $C^2$ estimate of the contracting flow give that $\lambda_i\Theta^{-1}\to 1$ as $t\to T$, which then leads to the $C^2$ estimate of $\omega$. Since $\omega$ satisfies
\begin{equation}\label{s6:w-t}
   \frac{\partial}{\partial \tau}\omega=-\frac 12vF^{-1}\tan(e^{-\tau})-\omega
\end{equation}
with respect to the new time parameter $\tau=-\ln\Theta(t,T)$, it is a direct calculation to check that \eqref{s6:w-t} is a uniformly parabolic equation of $\omega$. Since we already have $C^0,C^1$ and $C^2$ estimates of $\omega$, the regularity estimate of Andrews \cite{andrews2004fully} gives the $C^{2,\alpha}$ estimate of $\omega$ and the Schauder theory gives the $C^{m,\alpha}$ estimate of $\omega$ for all $m\geq 2$. Therefore, applying the interpolation inequalities (cf. Lemma 6.1 of \cite{gerhardt2011-H^n}) for the $C^m$ norms of $\omega$, we have that $\omega$ converges in $C^{\infty}$ to $1$ exponentially. This completes the proof of Theorem \ref{thm1.3}.

\section{Proof of Theorem \ref{thm-ex-H}}\label{sec:HW}

We define the quantity for a smooth surface $M$ in $\mathbb{H}^3$ by
\begin{equation*}
  Q(M)=-|M|\int_{M}|\mathring{A}|^2d\mu,
\end{equation*}
where $\mathring{A}$ is the trace-less part of the second fundamental form of $M$. The following proposition characterizes when the limit of rescaled metric is a round metric in terms of the value of $Q$.
\begin{prop}[\cite{Hung-Wang2015inverse}]\label{prop-hung-wang}
Let $\tilde{M}_s$ be a family of surfaces in $\mathbb{H}^3$ that are radial graphs of the function $u(s,\theta)=cs+f(\theta)+o(1)$ and $g_{ij}$ be the induced metric on $\tilde{M}_s$. Then the limit of the rescaled metric $e^{-2cs}g_{ij}$ as $s\to \infty$ is a round metric if and only if $\lim_{s\to\infty}Q(\tilde{M_s})=0$.
\end{prop}

Given an initial star-shaped and mean convex surface $M_0$ in $\mathbb{H}^3$, we consider the smooth solution $M_t$ of the following flow with $0<\alpha<1$:
\begin{equation}\label{flow-H^k}
\left\{\begin{aligned}
  \frac{\partial }{\partial t} X(x,t)=& H^{-\alpha}(x,t)\nu(x,t), \\
  X(\cdot,0)= & X_0(\cdot).
\end{aligned}\right.
\end{equation}

\begin{lem}
Suppose that $M_t$ is a smooth solution of the flow \eqref{flow-H^k}. We have that
\begin{equation}\label{Q-evl}
\begin{aligned}
  \frac{d}{dt}Q(M_t) =& ~|M_t|\int_{M_t}H^{1-\alpha}|\mathring{A}|^2d\mu_t-\int_{M_t}H^{1-\alpha}\int_{M_t}|\mathring{A}|^2d\mu_t\\
   & \quad +\alpha|M_t|\int_{M_t}H^{-1-\alpha}|\nabla H|^2d\mu_t.
\end{aligned}
\end{equation}
\end{lem}
\proof
Along the flow \eqref{flow-H^k}, it follows from \eqref{2.2} and \eqref{2.4} that the second fundamental form $h_i^j$ of $M_t$ evolves by
\begin{align*}
  \frac{\partial}{\partial t} h_i^j=& -\nabla^j\nabla_iH^{-\alpha}-H^{-\alpha}h_i^kh_k^j+H^{-\alpha}\delta_i^j \\
  = & -\nabla^j\nabla_iH^{-\alpha}-H^{-\alpha}\mathring{h}_i^k\mathring{h}_k^j-H^{1-\alpha}\mathring{h}_i^j+H^{-\alpha}(1-\frac {H^2}4)\delta_i^j,
\end{align*}
where $\mathring{h}_i^j=h_i^j-\frac H2\delta_i^j$ is the trace-less part of the second fundamental form. Then we have that
\begin{align*}
  \frac{\partial}{\partial t} \mathring{h}_i^j= -\nabla^j\nabla_iH^{-\alpha}-H^{-\alpha}\mathring{h}_i^k\mathring{h}_k^j-H^{1-\alpha}\mathring{h}_i^j+H^{-\alpha}(1-\frac {H^2}4)\delta_i^j-\frac 12\partial_tH\delta_i^j.
\end{align*}
Hence, we obtain that
\begin{align}\label{A-evl}
  \frac{\partial}{\partial t}|\mathring{A}|^2 =&  -2\mathring{h}_j^i\nabla^j\nabla_iH^{-\alpha}-2H^{1-\alpha}|\mathring{A}|^2,
\end{align}
where we used the facts $\mathring{h}_i^k\mathring{h}_k^j\mathring{h}_{j}^i=0$ and $\mathring{h}_j^i\delta_i^j=0$. Since the area form $d\mu_t$ evolves by
\begin{equation}\label{vol-evl}
  \frac{\partial}{\partial t}d\mu_t=H^{1-\alpha}d\mu_t,
\end{equation}
we have
\begin{align}\label{A-ev2}
  \frac{d}{dt}\int_{M_t}|\mathring{A}|^2 d\mu_t=&  -2\int_{M_t}\mathring{h}_j^i\nabla^j\nabla_iH^{-\alpha}d\mu_t-\int_{M_t}H^{1-\alpha}|\mathring{A}|^2d\mu_t\nonumber\\
  =&2\int_{M_t}\nabla^j\mathring{h}_j^i\nabla_iH^{-\alpha}d\mu_t-\int_{M_t}H^{1-\alpha}|\mathring{A}|^2d\mu_t\nonumber\\
  =&-\alpha\int_{M_t}H^{-1-\alpha}|\nabla H|^2d\mu_t-\int_{M_t}H^{1-\alpha}|\mathring{A}|^2d\mu_t,
\end{align}
where in the last equality we used the Codazzi equation to obtain
\begin{equation*}
  \nabla^j\mathring{h}_j^i=\nabla^jh_j^i-\frac 12\nabla^iH=\frac 12\nabla^iH.
\end{equation*}
The evolution equation \eqref{Q-evl} follows by combining \eqref{A-ev2} and \eqref{vol-evl}.
\endproof

We now describe the procedure of constructing the example in Theorem \ref{thm-ex-H}. Choose a function $\bar{f}(\theta)$ on $\mathbb{S}^2$ such that
\begin{equation*}
  \int_{\mathbb{S}^2}e^{2\bar{f}}d\mu_{g_{\mathbb{S}^2}}\int_{\mathbb{S}^2}|\mathring{D}^2e^{-\bar{f}}|_{g_{\mathbb{S}^2}}^2d\mu_{g_{\mathbb{S}^2}}=c_0>0,
\end{equation*}
where $\mathring{D}^2e^{-\bar{f}}$ means the traceless part of the Hessian of $e^{-\bar{f}}$, and $g_{\mathbb{S}^2}$ means the round metric on $\mathbb{S}^2$.
Recall that the metric on $\mathbb{H}^3$ can be expressed as
\begin{equation*}
\bar{g}=dr^2+\sinh^2r g_{\mathbb{S}^2}
\end{equation*}
for the $(r,\theta)$ coordinates centered at some point $x_0\in\mathbb{H}^3$.
Let $\tilde{M}_s$ be the family of surfaces given by the radial graph of $u(s,\theta)=s+\bar{f}(\theta)$ over $\mathbb{S}^2$ in the $(r,\theta)$ coordinates of $\mathbb{H}^3$. We know from \cite{Hung-Wang2015inverse} that
\begin{equation}\label{Qs-lim1}
  \lim_{s\to\infty}Q(\tilde{M}_s)=-c_0.
\end{equation}

\begin{prop}[\cite{Neves-2010}]
There exists a constant $C_0$ depending only on the bound
\begin{equation*}
  E:=\sup_{\mathbb{S}^2}\sum_{i=0}^3|\nabla^i\bar{f}|_{g_{\mathbb{S}^2}}
\end{equation*}
such that for all $s\geq 1$, the mean curvature $H$ and the second fundamental form $A$ of $\tilde{M}_s$ satisfy that
\begin{align}\label{neves-est1}
  |\tilde{M}_s||H-2|+|\tilde{M}_s||\mathring{A}|+ |\tilde{M}_s|^3|\nabla A|^2 \leq&~C_0.
\end{align}
\end{prop}

By choosing  $s_0>0$ large enough and using the estimate \eqref{neves-est1}, we can make sure that there exists a constant $\epsilon_0>0$ such that for any $s\geq s_0$, we have
\begin{equation}\label{Ms-condi}
  3\geq H\geq \epsilon_0, \quad \bar{g}(\partial_r,\nu)\geq \epsilon_0, \quad |\mathring{A}|^2< \frac 14H^2.
\end{equation}
Then $\tilde{M}_s$ is mean-convex and star-shaped for any $s\geq s_0$. Denote the in-radius and out-radius of the initial surface $\tilde{M}_s=\text{graph} ~ u(s,\theta)(=s+\bar{f}(\theta))$ in the $(r,\theta)$ coordinates of $\mathbb{H}^3$ by $\underline{r}_0$ and $\bar{r}_0$, i.e., $\underline{r}_0=\inf u(s,\theta),~\bar{r}_0=\sup u(s,\theta)$. Clearly, $\bar{r}_0-\underline{r}_0=\text{osc}(\bar{f})$ is independent of $s$. We consider the solution $M_t^s$ of the flow \eqref{flow-H^k} starting from $\tilde{M}_s$, where $t$ is the time parameter.
\begin{prop}\label{lem-decay}
There exist constants $s_0=s_0(\alpha,\epsilon_0,\bar{r}_0-\underline{r}_0, E)$  and $C=C(\alpha,\epsilon_0,\bar{r}_0-\underline{r}_0,E)$ depending only on $\alpha,\epsilon_0,\bar{r}_0-\underline{r}_0$, and $E$ such that for any $s\geq s_0$, we have
\begin{equation}\label{est-decay}
\begin{aligned}
  |\tilde{M}_s|^2|H-2|^2+|\tilde{M}_s|^2|\mathring{A}|^2\leq~&Ce^{-2^{(2-\alpha)}\cdot t}\\
  |\tilde{M}_s|^3|\nabla A|^2 \leq~&Ce^{-3\cdot 2^{(1-\alpha)}\cdot t}
\end{aligned}
\end{equation}
on the solution $M_t^s$ of the flow \eqref{flow-H^k} starting from $\tilde{M}_s$.
\end{prop}
The key point of the estimates \eqref{est-decay} is that the constant $C$ is independent of the parameter $s$. The proof of \eqref{est-decay} is technical and will be given in the next subsection.

We now complete the proof of Theorem \ref{thm-ex-H} using the estimate \eqref{est-decay}. By \eqref{Q-evl}, we have
\begin{align*}
\frac{d}{dt}Q(M_t^s) =& ~|M_t^s|\int_{M_t^s}|\mathring{A}|^2\left(H^{1-\alpha}-\frac 1{|M_t^s|}\int_{M_t^s}H^{1-\alpha}\right)d\mu_t\\
   & \quad +\alpha|M_t^s|\int_{M_t^s}H^{-1-\alpha}|\nabla H|^2d\mu_t.
\end{align*}
Since the volume element evolves by \eqref{vol-evl}, then using the estimate \eqref{est-decay} we have
\begin{equation*}
  |M_t^s|\leq \tilde{C}|\tilde{M}_s|e^{2^{(1-\alpha)}\cdot t}
\end{equation*}
and
\begin{align}\label{Qs-evl2}
\frac{d}{dt}Q(M_t^s) \leq & ~\tilde{C}|\tilde{M}_s|^{-1}e^{-2^{(1-\alpha)} \cdot t},
\end{align}
where $\tilde{C}$ is a constant independent of $s$. In view of \eqref{Qs-lim1}, we can choose $s_0$  large enough such that for any $s\geq s_0$,
\begin{equation*}
  Q(\tilde{M}_s)<-\frac{c_0}2,\quad \tilde{C} |\tilde{M}_s|^{-1}\leq 2^{-1-\alpha}c_0.
\end{equation*}
Integrating \eqref{Qs-evl2} gives that
\begin{equation}\label{Qs-lim2}
  \lim_{t\to\infty}Q(M_t^s)\leq -\frac{c_0}4<0.
\end{equation}
By the convergence result of Theorem  1.2 in \cite{Scheuer2015-1}, the solution $M_t^s$ exists for all time and is given by the radial graph of the function
\begin{equation*}
  u(t,\theta)=\frac t{2^{\alpha}}+f(\theta)+o(1),\quad\text{as } t\to\infty,
\end{equation*}
where $f(\theta)$ is a smooth function on $\mathbb{S}^2$. Then by combining \eqref{Qs-lim2} and Proposition \ref{prop-hung-wang}, we obtain  that the limit of the rescaled metric $e^{-2^{(1-\alpha)}\cdot t}g_t$ is not a round metric on $\mathbb{S}^2$. Note that
\begin{equation*}
 \lim_{t\to\infty} e^{-2^{(1-\alpha)}\cdot t}g_t=\lim_{t\to\infty} e^{-2^{(1-\alpha)}\cdot t}\sinh^2 u(t,\theta) g_{\mathbb{S}^2}=\frac 14e^{2f(\theta)}g_{\mathbb{S}^2}.
\end{equation*}
The fact that $e^{2f(\theta)}g_{\mathbb{S}^2}$ is not a round metric implies that $e^{-f(\theta)}$ is not a linear combination of constants and first eigenfunctions of $\mathbb{S}^2$ (cf. Lemma 4 in \cite{Hung-Wang2015inverse}).

\subsection{Proof of Proposition \ref{lem-decay}}\label{sec:7-1}

For the simplicity of the notations, in this subsection, we write $\tilde{M}_s$ as $M_0$, and $M_t^s$ as $M_t$.
First, we have the following evolution equations for $g, \chi, H$ and $A$, which follow from direct calculation using \eqref{2.2}, \eqref{2.4} and \eqref{s2:evl-chi}.
\begin{lem}
Along the flow \eqref{flow-H^k}, we have
\begin{itemize}
\item[(1)] The induced metric on $M_t$ evolves by
\begin{equation}\label{flow-H^k-g}
  \partial_tg_{ij}~=~2H^{-\alpha}h_{ij}.
\end{equation}
  \item[(2)] The support function $\chi=\bar{g}(\sinh u\partial_r,\nu)$ satisfies
  \begin{equation}\label{flow-H^k-chi}
    \partial_t\chi~=~\alpha H^{-\alpha-1}\Delta\chi+\alpha H^{-\alpha-1}|A|^2\chi+(1-\alpha)H^{-\alpha}\cosh u.
      \end{equation}
  \item[(3)] The mean curvature $H$ satisfies
  \begin{equation}\label{flow-H^k-H}
  \begin{aligned}
    \partial_tH=& \alpha H^{-\alpha-1}\Delta H-\alpha(\alpha+1)H^{-\alpha-2}|\nabla H|^2-|A|^2H^{-\alpha}+2H^{-\alpha}\\
   =&-\Delta H^{-\alpha}-|A|^2H^{-\alpha}+2H^{-\alpha}.
    \end{aligned}
  \end{equation}
  \item[(4)] The second fundamental form $A=(h_{ij})$ satisfies
  \begin{equation}\label{flow-H^k-A0}
   \begin{aligned}
    \partial_th_{ij}~=~&\alpha H^{-\alpha-1}\Delta h_{ij}-\alpha(\alpha+1)H^{-\alpha-2}\nabla_iH\nabla_jH+\alpha H^{-\alpha-1}|A|^2h_{ij}\\
    &+(1-\alpha)H^{-\alpha}h_i^kh_{kj}+(1-\alpha)H^{-\alpha}g_{ij}+2\alpha H^{-\alpha-1}h_{ij},
    \end{aligned}
  \end{equation}
  \begin{equation}\label{flow-H^k-A}
    \begin{aligned}
    \partial_t|A|^2~=~&\alpha H^{-\alpha-1}\Delta |A|^2-2\alpha H^{-\alpha-1}|\nabla A|^2-2\alpha(\alpha+1)H^{-\alpha-2}h_{ij}\nabla^iH\nabla^jH\\
    &+2\alpha H^{-\alpha-1}|A|^4-2(1+\alpha)H^{-\alpha}h_i^kh_k^jh_j^i\\
    &+2(1-\alpha)H^{-\alpha+1}+4\alpha H^{-\alpha-1}|A|^2.
    \end{aligned}
  \end{equation}
\end{itemize}
\end{lem}

\begin{lem}\label{lem6.6}
 Let $\tilde{\chi}=e^{-\frac t{2^{\alpha}}}\chi$. Then for each $t\geq0$, we have
\begin{equation}\label{lem-7-4-ineqn}
  \frac 12e^{\bar{r}_0}\geq \tilde{\chi}(x,t)\geq \min \tilde{\chi}(\cdot,0)\geq \epsilon_0\sinh\underline{r}_0
\end{equation}
on the solution $M_t$ of \eqref{flow-H^k}.
\end{lem}
\proof
First, since $\alpha\in (0,1)$, we know from \eqref{flow-H^k-chi} that $\tilde{\chi}=e^{-\frac t{2^{\alpha}}}\chi>0$ is preserved along the flow.  By using \eqref{flow-H^k-chi}, we have that
\begin{equation}\label{flow-H^k-chi-td}
    \partial_t\tilde{\chi}~=~\alpha H^{-\alpha-1}\Delta\tilde{\chi}+\alpha H^{-\alpha-1}|A|^2\tilde{\chi}+(1-\alpha)H^{-\alpha}e^{-\frac t{2^{\alpha}}}\cosh u-\frac 1{2^{\alpha}}\tilde{\chi}.
      \end{equation}
Since
\begin{equation*}
  \cosh u\geq \sinh u\geq \chi,\quad |A|^2\geq H^2/2,\quad \alpha\in(0,1),
\end{equation*}
\eqref{flow-H^k-chi-td} implies that
\begin{equation}\label{flow-H^k-chi-td2}
    \partial_t\tilde{\chi}~\geq~\alpha H^{-\alpha-1}\Delta\tilde{\chi}+\tilde{\chi}\left(\frac{\alpha} 2 H^{-\alpha+1}+(1-\alpha)H^{-\alpha}-\frac 1{2^{\alpha}}\right).
      \end{equation}
We claim that the terms in the bracket of \eqref{flow-H^k-chi-td2} are always nonnegative for $\forall~H>0, \forall~\alpha\in(0,1)$. In fact, for each fixed $\alpha\in(0,1)$, the minimum of
\begin{equation*}
  \frac{\alpha} 2 H^{-\alpha+1}+(1-\alpha)H^{-\alpha}-\frac 1{2^{\alpha}}
\end{equation*}
is achieved at $H=2$. Therefore,
\begin{equation*}
  \frac{\alpha} 2 H^{-\alpha+1}+(1-\alpha)H^{-\alpha}-\frac 1{2^{\alpha}}~\geq  ~\frac{\alpha} 2 \cdot2^{-\alpha+1}+(1-\alpha)2^{-\alpha}-\frac 1{2^{\alpha}}=0.
\end{equation*}
 Since $\tilde{\chi}$ is always positive, \eqref{flow-H^k-chi-td2} implies that
\begin{equation*}
  \tilde{\chi}(x,t)~\geq~\min \tilde{\chi}(\cdot,0)\geq \epsilon_0\sinh\underline{r}_0.
\end{equation*}
On the other hand, by using \eqref{rho-t-H} and \eqref{ut-H-C0}, we know that  $u(t)\leq \rho(t,\bar{r}_0)\leq \bar{r}_0+\frac t{2^{\alpha}}$, so we obtain that
\begin{equation*}
  \tilde{\chi}(x,t)=e^{-\frac t{2^{\alpha}}}\sinh u(t) v^{-1}\leq \frac 12 e^{u(t)-\frac t{2^{\alpha}}} v^{-1}\leq \frac {e^{\bar{r}_0}}2.
\end{equation*}
\endproof

\begin{lem}\label{lem-ap-H}
The mean curvature $H$ of $M_t$ is bounded from below by a positive constant $C_1=C_1(\alpha,\epsilon_0,\bar{r}_0-\underline{r}_0)$ depending only on $\alpha,\epsilon_0$, and $\bar{r}_0-\underline{r}_0$.
\end{lem}
\proof
Let $\zeta_t(x)=\tilde{\chi}^{-1}(x,t)H^{-\alpha}(x,t)$, where $\tilde{\chi}=e^{-\frac t{2^{\alpha}}}\chi$ is defined in Lemma \ref{lem6.6}. Combining \eqref{flow-H^k-H} and \eqref{flow-H^k-chi-td} yields
\begin{align*}
  \partial_t\zeta_t =~& \alpha H^{-\alpha-1}\Delta\zeta_t+2\alpha H^{-\alpha-1}\nabla\zeta_t \cdot\nabla\ln\tilde{\chi}\\
   & +(\alpha-1)\zeta_t^2e^{-\frac t{2^{\alpha}}}\cosh u-2\alpha H^{-\alpha-1}\zeta_t+\frac 1{2^{\alpha}}\zeta_t\\
   \leq~&\alpha H^{-\alpha-1}\Delta\zeta_t+2\alpha H^{-\alpha-1}\nabla\ln\tilde{\chi}\cdot\nabla\zeta_t+\left(-2\alpha \zeta_t^{\frac{\alpha+1}{\alpha}}\tilde{\chi}^{\frac{\alpha+1}{\alpha}}+\frac 1{2^{\alpha}}\right)\zeta_t,
\end{align*}
where we used the assumption $0<\alpha<1$. Then applying the maximum principle to the differential inequality  above, we conclude that
\begin{align*}
  \zeta_t(x)\leq ~&\max\left\{\max_{M_0}\zeta_0, 2^{-\alpha}\alpha^{-\frac{\alpha}{\alpha+1}}(\min \tilde{\chi})^{-1}\right\}\nonumber\\
  \leq ~& (\min \tilde{\chi}(\cdot,0))^{-1}\max\left\{\epsilon_0^{-\alpha},2^{-\alpha}\alpha^{-\frac{\alpha}{\alpha+1}}\right\},
\end{align*}
where $\epsilon_0$ is the same constant with that in \eqref{Ms-condi}.
By the definition of $\zeta_t$ and the estimate in Lemma \ref{lem6.6}, we have that
\begin{align}\label{s7-H-lbd}
  H(x,t)\geq ~&(\frac{\max\tilde{\chi}(\cdot,t)}{\min \tilde{\chi}(\cdot,0)})^{-\frac 1{\alpha}}\left(\max\{\epsilon_0^{-\alpha},2^{-\alpha}\alpha^{-\frac{\alpha}{\alpha+1}}\}\right)^{-\frac 1{\alpha}}\nonumber\\
  \geq&\epsilon_0^{\frac 1{\alpha}}e^{-\frac 1{\alpha}(\bar{r}_0-\underline{r}_0)}\left(1-e^{-2\underline{r}_0}\right)^{\frac 1{\alpha}}\cdot \min\{\epsilon_0,2\alpha^{\frac{1}{\alpha+1}}\}\nonumber\\
  \geq &2^{-\frac 1{\alpha}}\epsilon_0^{\frac 1{\alpha}}e^{-\frac 1{\alpha}(\bar{r}_0-\underline{r}_0)}\cdot \min\{\epsilon_0,2\alpha^{\frac{1}{\alpha+1}}\},
\end{align}
since $s_0$ can be  chosen large enough in \eqref{Ms-condi} and then $1-e^{-2\underline{r}_0}$ will  be greater than $1/2$. Then the lemma follows by defining $C_1$ as the right hand side of \eqref{s7-H-lbd}.
\endproof

In the following, we prove Proposition  \ref{lem-decay} by proving Lemma \ref{lem6.8} and Lemma \ref{lem6.9} (Note that in the proof of Lemma \ref{lem6.8}, we obtain that $H$ is bounded from below and above by positive constants which do not depend on $s$; hence, $|H-2|$ and $|H^2-4|$ have the same decay rates).
We note that in the proof of Lemma \ref{lem6.8} and Lemma \ref{lem6.9}, we always choose $s_0$ large enough such that $|M_0|>1$.

\begin{lem}\label{lem6.8}
There exist constants $s_0=s_0(\alpha,\epsilon_0,\bar{r}_0-\underline{r}_0, E)$ and $C_2=C_2(\alpha,\epsilon_0,\bar{r}_0-\underline{r}_0,E)$ such that for any $s\geq s_0$,
\begin{align}
  |M_0|^2|\mathring{A}|^2(x,t)\leq~ &C_2e^{-2^{(2-\alpha)}\cdot t},  \label{s7-As-decay}\\
  |M_0||H^2(x,t)-4| \leq~& C_2e^{-2^{(1-\alpha)}\cdot t}\label{s7-H-decay0}
\end{align}
on $M_t$ along the flow \eqref{flow-H^k}.
\end{lem}
\proof
We define $\gamma_t=|\mathring{A}|^2H^{-2}$, then $\gamma_t=|A|^2H^{-2}-\frac 12$. Combining \eqref{flow-H^k-H} and \eqref{flow-H^k-A}, we have
\begin{align}\label{s7-A-1}
  \partial_t(|A|^2H^{-2}) =& \alpha H^{-\alpha-1}\Delta(|A|^2H^{-2})+2(\alpha+1)H^{1-\alpha}(|A|^2H^{-2})^2\nonumber\\
   &+4(\alpha-1)H^{-\alpha-1}(|A|^2H^{-2}) +2(1-\alpha)H^{-\alpha-1}\nonumber\\
   &-2(\alpha+1)H^{-\alpha-2}h_i^kh_k^jh_j^i+Q,
\end{align}
where $Q$ is the gradient terms
\begin{align*}
  Q= & 4\alpha H^{-\alpha-4}\nabla |A|^2\cdot\nabla H-2\alpha H^{-\alpha-3}|\nabla A|^2 \\
   &-2\alpha(\alpha+1)H^{-\alpha-4}h_{ij}\nabla^iH\nabla^jH+2\alpha(\alpha-2)H^{-\alpha-5}|A|^2|\nabla H|^2.
\end{align*}
As the dimension of $M_t$ is $2$, then
\begin{align}\label{s7-A3}
  h_i^kh_k^jh_j^i= & \sum_{i=1}^2\lambda_i^3=\sum_{i=1}^2(\mathring{\lambda}_i+\frac H2)^3=\frac 32|\mathring{A}|^2H+\frac{H^3}4,
\end{align}
where $\mathring{\lambda}_1,\mathring{\lambda}_2$ are the eigenvalues of $\mathring{A}$.
Substituting $|A|^2H^{-2}=\gamma_t+\frac 12$ and \eqref{s7-A3} into \eqref{s7-A-1}, we have
\begin{align}\label{gammat-1}
  \partial_t\gamma_t =~& \alpha H^{-\alpha-1}\Delta\gamma_t+2(\alpha+1)H^{1-\alpha}\gamma_t(\gamma_t-\frac 12)+4(\alpha-1)H^{-\alpha-1}\gamma_t+Q.
\end{align}
We denote by $\{\nu_1,\nu_2\}$ an orthonormal eigenbasis for $\mathring{A}$ at the critical point $p$ of $\gamma_t$, and assume that $\mathring{A}(\nu_1,\nu_2)\geq 0$ without loss
of generality. Then at $p$ we have (cf. \cite[p.208]{Neves-2010})
\begin{equation}
\begin{aligned}
  &\nabla |\mathring{A}|^2=2|\mathring{A}|^2H^{-1}\nabla H,~\nabla|A|^2=2|A|^2H^{-1}\nabla H,\\
  &|\mathring{A}|H^{-1}\nabla H=\sqrt{2}\nabla\mathring{A}(\nu_1,\nu_1)=-\sqrt{2}\nabla\mathring{A}(\nu_2,\nu_2),\\
  &2|\nabla\mathring{A}(\nu_1,\nu_2)|^2=(\gamma_t+\frac{1}{2})|\nabla H|^2-\frac{2}{H}\mathring{h}_{ij}\nabla^iH\nabla^jH.
  \end{aligned}
\end{equation}
It follows that at $p$, we have
\begin{equation*}
\begin{aligned}
|\nabla \mathring{A}|^2&=|\mathring{A}|^2H^{-2}|\nabla H|^2+2|\nabla\mathring{A}(\nu_1,\nu_2)|^2\\
&=(2\gamma_t+\frac{1}{2})|\nabla H|^2-\frac{2}{H}\mathring{h}_{ij}\nabla^iH\nabla^jH
  \end{aligned}
\end{equation*}
and
\begin{equation*}
\begin{aligned}
|\nabla A|^2&=|\nabla \mathring{A}|^2+\frac{1}{2}|\nabla H|^2\\
&=(2\gamma_t+1)|\nabla H|^2-\frac{2}{H}\mathring{h}_{ij}\nabla^iH\nabla^jH.
  \end{aligned}
\end{equation*}
Then at $p$, the gradient term $Q$ satisfies
\begin{equation}
\begin{aligned}\label{gammat-2}
 Q= ~& 4\alpha H^{-\alpha-4}\cdot2|A|^2H^{-1}|\nabla H|^2-2\alpha H^{-\alpha-3}\Big((2\gamma_t+1)|\nabla H|^2-\frac{2}{H}\mathring{h}_{ij}\nabla^iH\nabla^jH\Big)\\
 &-2\alpha(\alpha+1)H^{-\alpha-4}(\mathring{h}_{ij}+\frac{H}{2}g_{ij})\nabla^iH\nabla^jH+2\alpha(\alpha-2)H^{-\alpha-5}|A|^2|\nabla H|^2 \\
  =~ & 2\alpha^2H^{-\alpha-3}|\nabla H|^2\gamma_t-\alpha H^{-\alpha-3}|\nabla H|^2  +2\alpha(1-\alpha)H^{-\alpha-4}\mathring{h}_{ij}\nabla^iH\nabla^jH\\
   \leq ~&2\alpha(\alpha+(1-\alpha)^2)H^{-\alpha-3}|\nabla H|^2\gamma_t-\frac {\alpha}2H^{-\alpha-3}|\nabla H|^2\\
   \leq ~&2\alpha H^{-\alpha-3}|\nabla H|^2(\gamma_t-\frac 14),
\end{aligned}
\end{equation}
where  $0<1-\alpha<1$ is used in the last inequality and the Cauchy-Schwartz inequality
$$(1-\alpha)H^{-1}\mathring{h}_{ij}\nabla^iH\nabla^jH\leq (1-\alpha)^2H^{-2}|\mathring{A}|^2|\nabla H|^2+\frac{1}{4}|\nabla H|^2$$
is used in the first inequality.

As we assumed in \eqref{Ms-condi} that the initial surface satisfies $\gamma_0< 1/4$, by \eqref{gammat-1} and \eqref{gammat-2} and noting that $\alpha\leq 1$, we have that $\gamma_t<1/4$ for each $t>0$. Moreover, using $\gamma_t<1/4$, Lemma \ref{lem-ap-H} and applying the maximum principle to \eqref{gammat-1} and \eqref{gammat-2} again, we have that
\begin{equation}\label{s7-gam-decay1}
  \gamma_t~\leq ~(\max_{M_0}\gamma_0) e^{-\delta_1 t},
\end{equation}
where $0<\delta_1= \frac{\alpha+1}2C_1^{1-\alpha}$ is a positive constant depending only on $\alpha,\epsilon_0$,and $\bar{r}_0-\underline{r}_0$, i.e., $\gamma_t\to 0$ exponentially as $t\to \infty$ in the rate $\delta_1$.
In the following, we improve the rate of the exponential decay in \eqref{s7-gam-decay1} step by step.

We define another function on $M_t$ by
\begin{equation}\label{s7-phi-def}
  \phi_t=e^{2^{(1-\alpha)}\cdot t}(H^2-4).
\end{equation}
 By  using \eqref{flow-H^k-H}, we have that
\begin{equation}\label{s7-phi-evl}
\begin{aligned}
  \partial_t\phi_t =~&\alpha H^{-\alpha-1}\Delta \phi_t-\alpha(\alpha+2)H^{-\alpha-2}\nabla\phi_t\cdot\nabla H\\
  &\quad -2H^{1-\alpha}\left(|A|^2-2\right)e^{2^{(1-\alpha)}\cdot t}+2^{1-\alpha}\phi_t\\
  \leq ~&\alpha H^{-\alpha-1}\Delta \phi_t-\alpha(\alpha+2)H^{-\alpha-2}\nabla\phi_t\cdot\nabla H-\left(H^{1-\alpha}-2^{1-\alpha}\right)\phi_t\\
  \leq ~&\alpha H^{-\alpha-1}\Delta \phi_t-\alpha(\alpha+2)H^{-\alpha-2}\nabla\phi_t\cdot\nabla H,
\end{aligned}
\end{equation}
where we used the facts that $|A|^2\geq H^2/2$ and $H^{1-\alpha}-2^{1-\alpha}$ has the same sign with $\phi_t$. Applying the maximum principle to \eqref{s7-phi-evl}, we obtain that
\begin{equation}\label{s7-H-decay1}
  H^2\leq 4+(\max_{M_0}|H^2-4|)e^{-2^{(1-\alpha)}\cdot t}.
\end{equation}
Substituting \eqref{s7-gam-decay1} into \eqref{gammat-1}, using the estimate \eqref{gammat-2} and \eqref{s7-H-decay1}, we have that at the critical point of $\gamma_t$,
\begin{align}\label{gammat-3}
  \partial_t\gamma_t \leq~& \alpha H^{-\alpha-1}\Delta\gamma_t-\left(2-2(\alpha+1)(\max_{M_0}\gamma_0)e^{-\delta_1t}-\frac{1-\alpha}{4}(\max_{M_0}|H^2-4|)e^{-2^{(1-\alpha)}\cdot t}\right)H^{1-\alpha}\gamma_t.
\end{align}
Then the maximum principle implies that
\begin{align}\label{s7-gam-decay2}
  \gamma_t\leq~ & (\max_{M_0}\gamma_0)e^{-\delta_2t}\exp\left(\frac{2(\alpha+1)}{\delta_1}(\max_{M_0}\gamma_0)(\max H^{1-\alpha})+2^{\alpha-3}(1-\alpha)(\max_{M_0}|H^2-4|)(\max H^{1-\alpha})\right)\nonumber\\
  =:~&C_3(\max_{M_0}\gamma_0)e^{-\delta_2t},
\end{align}
where $C_3=C_3(\alpha,\epsilon_0,\bar{r}_0-\underline{r}_0)$, and $\delta_2=2C_1^{1-\alpha}>\delta_1$ are both positive constants depending only on $\alpha,\epsilon_0$,and $\bar{r}_0-\underline{r}_0$.

We now improve the lower bound of $H$ in Lemma \ref{lem-ap-H}.  From the evolution equation \eqref{flow-H^k-H}, we have
\begin{align*}
  \partial_tH^{1-\alpha}=& \alpha H^{-\alpha-1}\Delta H^{1-\alpha}-\alpha(1-\alpha)H^{-2\alpha-2}|\nabla H|^2\\
  &-(1-\alpha)H^{-2\alpha}|\mathring{A}|^2-\frac{1-\alpha}2H^{-2\alpha}(H^2-4).
\end{align*}
We consider the minimum point $p$ of $H$ on $M_t$. Our aim is to show that there exists a constant $C_4(\alpha,\epsilon_0,\bar{r}_0-\underline{r}_0)$ such that $H^{1-\alpha}\geq 2^{1-\alpha}-C_4 e^{-C_1^{(1-\alpha)}\cdot t}$ with $C_1$ the lower bound given in Lemma \ref{lem-ap-H}, then we can improve the estimate in \eqref{gammat-3}. If $H(p,t)\geq 2$, then we are done. If $H(p,t)<2$, then using the elementary inequality
\begin{equation*}
  (1+x)^a\geq 1+ax,\quad \forall ~x>0,~a>1,
\end{equation*}
we have
\begin{align*}
  4-H^2 =& H^2\left(1+H^{\alpha-1}(2^{1-\alpha}-H^{1-\alpha})\right)^{\frac 2{1-\alpha}}-H^2 \\
   \geq & \frac 2{1-\alpha}H^{1+\alpha}(2^{1-\alpha}-H^{1-\alpha}),
\end{align*}
since $2^{1-\alpha}-H^{1-\alpha}>0$ and $\frac 2{1-\alpha}>1$. So at the point $p$, we have
\begin{align*}
  \partial_t(2^{1-\alpha}-H^{1-\alpha}) \leq & -H^{1-\alpha}(2^{1-\alpha}-H^{1-\alpha})+(1-\alpha)H^{-2\alpha}|\mathring{A}|^2\\
  \leq &-C_1^{1-\alpha}(2^{1-\alpha}-H^{1-\alpha})+(1-\alpha)(4+\max_{M_0}|H^2-4|)C_1^{-2\alpha}C_3(\max_{M_0}\gamma_0)e^{-\delta_2t},
\end{align*}
where we used the estimates \eqref{s7-H-decay1} and \eqref{s7-gam-decay2}. Then the maximum principle implies that
\begin{align}\label{s7-H-decay2}
 2^{1-\alpha}-H^{1-\alpha} \leq & e^{-C_1^{(1-\alpha)}\cdot t}\left(\max_{M_0}|2^{1-\alpha}-H^{1-\alpha}|+(1-\alpha)(4+\max_{M_0}|H^2-4|)C_1^{-\alpha-1}C_3\max_{M_0}\gamma_0\right)\nonumber\\
 =:&C_4e^{-C_1^{(1-\alpha)}\cdot t},
\end{align}
where $C_4=C_4(\alpha,\epsilon_0,\bar{r}_0-\underline{r}_0)$.

Applying the estimate \eqref{s7-H-decay2} in \eqref{gammat-3}, we have that at the critical point of $\gamma_t$,
\begin{align}\label{gammat-4}
  \partial_t\gamma_t \leq~& \alpha H^{-\alpha-1}\Delta\gamma_t-\left(2^{2-\alpha}-C_5(e^{-\delta_1t}+e^{-2^{(1-\alpha)}\cdot t}+e^{-C_1^{(1-\alpha)}\cdot t})\right)\gamma_t,
\end{align}
for some constant $C_5=C_5(\alpha,\epsilon_0,\bar{r}_0-\underline{r}_0)$. By applying the maximum principle again, we obtain that
\begin{align}\label{s7-gam-decay3}
  \gamma_t \leq~& \max_{M_0}\gamma_0e^{-2^{(2-\alpha)}\cdot t}\exp(C_5(\frac 1{\delta_1}+2^{\alpha-1}+C_1^{\alpha-1}))\nonumber \\
  =:~& C_6(\max_{M_0}\gamma_0)e^{-2^{(2-\alpha)}\cdot t}
\end{align}
with $C_6=C_6(\alpha,\epsilon_0,\bar{r}_0-\underline{r}_0)$. Since $|\mathring{A}|^2=\gamma_tH^2$, the estimate \eqref{s7-As-decay} follows from \eqref{s7-gam-decay3}, \eqref{s7-H-decay1}, \eqref{Ms-condi} and the estimate \eqref{neves-est1}.

In the following, we prove \eqref{s7-H-decay0}. Note that from \eqref{s7-H-decay1}, \eqref{neves-est1} and \eqref{Ms-condi}, we immediately get that
\begin{equation*}
   |M_0|(H^2(x,t)-4) \leq~ |M_0|(\max_{M_0}|H^2-4|)e^{-2^{(1-\alpha)}\cdot t}~\leq~ Ce^{-2^{(1-\alpha)}\cdot t}
\end{equation*}
with $C$ depending only on $E$. In order to prove \eqref{s7-H-decay0}, it remains to estimate the lower bound of $|M_0|(H^2(x,t)-4)$. We prove this by two steps.

From the definition of $C_4$ given in \eqref{s7-H-decay2} and the estimate given in \eqref{neves-est1},
 we can choose $s_0$ large enough such that $0<C_4\leq 2^{1-\alpha}-1$, then \eqref{s7-H-decay2} yields that
 \begin{equation*}
 H^{1-\alpha}\geq 2^{1-\alpha}-C_4e^{-C_1^{(1-\alpha)}\cdot t}\geq 2^{1-\alpha}-C_4\geq 2^{1-\alpha}-(2^{1-\alpha}-1)=1,
 \end{equation*}
 from which we defer that $H\geq 1$. This means that we can take $C_1=1$ as the lower bound of $H$. By taking $C_1=1$, choosing $s_0$ large enough and repeating  the same procedure as that to obtain \eqref{s7-H-decay2}, we can get
 \begin{equation}\label{H-improve}
 2^{1-\alpha}-H^{1-\alpha}\leq e^{-t}.
 \end{equation}
Let $\eta_t=e^t(H^2-4)$, we have
 \begin{equation*}
\begin{aligned}
  \partial_t\eta_t =~&\alpha H^{-\alpha-1}\Delta \eta_t-\alpha(\alpha+2)H^{-\alpha-2}\nabla\eta_t\cdot\nabla H\\
&\quad +\left(1-H^{1-\alpha}\right)\eta_t-2H^{1-\alpha}e^{t}|\mathring{A}|^2.
\end{aligned}
 \end{equation*}
 We consider the minimum point $p$ of $\eta_t$. If $\eta_t(p)\geq 0$, i.e., $H^2-4\geq0$, then we are done. If $\eta_t(p)<0$, then $(1-H^{1-\alpha})\eta_t(p)\geq 0$, so at $p$, using the estimate \eqref{s7-As-decay} and \eqref{s7-H-decay1}, we have
 \begin{equation*}
\begin{aligned}
  \partial_t\eta_t&\geq-2H^{1-\alpha}e^{t}|\mathring{A}|^2\geq-C_7|M_0|^{-2}e^{-(2^{2-\alpha}-1)t},
\end{aligned}
\end{equation*}
with $C_7$ depending only on $\alpha,\epsilon_0$, $\bar{r}_0-\underline{r}_0$ and $E$. Then applying the maximum principle, we obtain that
\begin{equation*}
\eta_t\geq\min_{M_0}\eta_t-\frac{C_7}{2^{2-\alpha}-1}|M_0|^{-2}\geq-\max_{M_0}|H^2-4|-\frac{C_7}{2^{2-\alpha}-1}|M_0|^{-2}.
\end{equation*}
We can choose $s_0$ large enough such that $|M_0|>1$, then from the estimates given in \eqref{neves-est1} and \eqref{Ms-condi},  we get
\begin{equation}\label{H-lower1}
\eta_t\geq-C_8(\alpha,\epsilon_0,\bar{r}_0-\underline{r}_0,E)|M_0|^{-1}.
\end{equation}

 Now we improve the estimate in \eqref{H-lower1}. For the function $\phi_t$ defined by \eqref{s7-phi-def}, we have
\begin{equation*}
\begin{aligned}
  \partial_t\phi_t
= ~&\alpha H^{-\alpha-1}\Delta \phi_t-\alpha(\alpha+2)H^{-\alpha-2}\nabla\phi_t\cdot\nabla H\\
&\quad +\left(2^{1-\alpha}-H^{1-\alpha}\right)\phi_t-2H^{1-\alpha}e^{2^{(1-\alpha)}\cdot t}|\mathring{A}|^2.
\end{aligned}
\end{equation*}
 We consider the minimum point $p$ of $\phi_t$. If $\phi_t(p)\geq 0$, i.e., $H\geq2$, then we are done. If $\phi_t(p)<0$, then at $p$, using \eqref{s7-As-decay}, \eqref{H-improve} and \eqref{H-lower1}, we have
\begin{equation}\label{s7-phi-evl-3}
\begin{aligned}
  \partial_t\phi_t&\geq~\left(2^{1-\alpha}-H^{1-\alpha}\right)\phi_t-2H^{1-\alpha}e^{2^{(1-\alpha)}\cdot t}|\mathring{A}|^2\\
  &=\left(2^{1-\alpha}-H^{1-\alpha}\right)e^{(2^{1-\alpha}-1)t}\eta_t-2H^{1-\alpha}e^{2^{(1-\alpha)}\cdot t}|\mathring{A}|^2\\
  &\geq -C_8 e^{-(2-2^{1-\alpha})t}|M_0|^{-1}-2^{2-\alpha}C_2e^{-2^{(1-\alpha)}\cdot t}|M_0|^{-2}.
\end{aligned}
\end{equation}
Then applying the maximum principle, we obtain that
\begin{equation*}
\phi_t\geq\min_{M_0}\phi_t-C|M_0|^{-1}-C|M_0|^{-2}\geq-\max_{M_0}|H^2-4|-C|M_0|^{-1}-C|M_0|^{-2},
\end{equation*}
for some $C$ depending only  on $\alpha,\epsilon_0$, $\bar{r}_0-\underline{r}_0$ and $E$.
From the estimates given in \eqref{neves-est1} and \eqref{Ms-condi},  we get
\begin{equation}\label{H-lower2}
\phi_t\geq-C(\alpha,\epsilon_0,\bar{r}_0-\underline{r}_0,E)|M_0|^{-1}.
\end{equation}
Then \eqref{s7-H-decay0} follows immediately.
\endproof

\begin{lem}\label{lem6.9}
There exist constants $s_0=s_0(\alpha,\epsilon_0,\bar{r}_0-\underline{r}_0, E)$ and $C_9=C_9(\alpha,\epsilon_0,\bar{r}_0-\underline{r}_0, E)$ such that for any $s\geq s_0$,
\begin{equation}\label{lem-s7-nA}
  |M_0|^3|\nabla A|^2(x,t)\leq C_9e^{-3\cdot 2^{(1-\alpha)}\cdot t}
\end{equation}
on $M_t$ along the flow \eqref{flow-H^k}.
\end{lem}
\proof
First, recall that in $2$-dimensional case, we have (cf. \cite[\S 2]{huisken-1984})
\begin{equation*}
  |\nabla H|^2\leq \frac 43|\nabla A|^2,
\end{equation*}
which implies  that
\begin{equation}\label{AAknot}
  |\nabla A|^2\leq 3|\nabla\mathring{A}|^2,\quad\text{and}\quad |\nabla H|^2\leq 4|\nabla\mathring{A}|^2.
\end{equation}
Therefore, in order to estimate $|\nabla A|^2$, it is equivalent to estimating $|\nabla\mathring{A}|^2$.

By combining \eqref{flow-H^k-g}, \eqref{flow-H^k-H} and \eqref{flow-H^k-A0}, we obtain that
\begin{align}\label{s7-nA-pf1}
  \partial_t\mathring{h}_{ij} =& \partial_th_{ij}-\frac 12\partial_tHg_{ij}-\frac H2\partial_tg_{ij} \nonumber\\
  = &\alpha H^{-\alpha-1}\Delta\mathring{h}_{ij}-\alpha(\alpha+1)H^{-\alpha-2}\nabla_iH\nabla_jH\nonumber\\
  &+\frac 12\alpha(\alpha+1)H^{-\alpha-2}|\nabla H|^2g_{ij}+\alpha H^{-\alpha-1}|\mathring{A}|^2\mathring{h}_{ij} +\frac{\alpha+1}2H^{-\alpha}|\mathring{A}|^2g_{ij}\nonumber\\
  &-\frac{\alpha} 2H^{1-\alpha}\mathring{h}_{ij}+(1-\alpha)H^{-\alpha}\mathring{h}_i^k\mathring{h}_{kj}+2\alpha H^{-\alpha-1}\mathring{h}_{ij}.
\end{align}

If $Rm$ denotes the curvature tensor of $M_t$, then for any tensor $S$ on $M_t$, the Ricci identity implies that
\begin{equation}\label{s7-nA-pf2}
  \Delta \nabla S=\nabla\Delta S+Rm*\nabla S+\nabla Rm*S,
\end{equation}
where $*$ means the contraction of two tensors using the metric $g(t)$ of $M_t$. We also have the formula of commuting $\nabla$ with $\partial_t$,
\begin{equation}\label{s7-na-pf3}
  \partial_t\nabla S-\nabla\partial_tS=S*\nabla\partial_tg(t).
\end{equation}
By Gauss equations, the curvature of $M_t$ satisfies
\begin{align*}
  R_{ijkl}=&-\left(g_{ik}g_{jl}-g_{il}g_{jk}\right)+h_{ik}h_{jl}-h_{il}h_{jk}\nonumber\\
  =&(\frac{H^2}4-1)\left(g_{ik}g_{jl}-g_{il}g_{jk}\right)+\mathring{h}_{ik}\mathring{h}_{jl}-\mathring{h}_{il}\mathring{h}_{jk}\nonumber\\
  &+\frac H2\left(\mathring{h}_{ik}g_{jl}+\mathring{h}_{jl}g_{ik}-\mathring{h}_{il}g_{jk}-\mathring{h}_{jk}g_{il}\right)
\end{align*}
in local coordinates.  Then
\begin{align}
  |Rm|\leq & C_{10}\left(|H^2-4|+|\mathring{A}|^2+H|\mathring{A}|\right)
\end{align}
and
\begin{align}
  |\nabla Rm|\leq  & C_{10}\left(H|\nabla H|+|\mathring{A}||\nabla\mathring{A}|+|\mathring{A}||\nabla H|+H|\nabla\mathring{A}|\right)\nonumber\\
  \leq & 3C_{10}\left(|\mathring{A}||\nabla\mathring{A}|+H|\nabla\mathring{A}|\right)\label{na-Rm2}
\end{align}
for a universal constant $C_{10}$ depending only on the dimension $n=2$.

Note that in the proof of Lemma \ref{lem6.8}, we obtain that $H$ and $|A|^2$ are bounded from below and above by positive constants which do not depend on $s$,
using \eqref{flow-H^k-g}, \eqref{AAknot}--\eqref{na-Rm2}, we compute and estimate the evolution equation of $|\nabla\mathring{A}|^2$.
\begin{align}\label{s7-na-pf4}
  \partial_t|\nabla\mathring{A}|^2 =& 2g(\partial_t\nabla\mathring{A},\nabla\mathring{A})-2\partial_tg_{rs}g^{ir}g^{ms}g^{jn}g^{kp}\mathring{h}_{ij,k}\mathring{h}_{mn,p}
  -\partial_tg_{rs}g^{kr}g^{ps}g^{im}g^{jn}\mathring{h}_{ij,k}\mathring{h}_{mn,p} \nonumber\\
  =&2g(\partial_t\nabla\mathring{A},\nabla\mathring{A})-4H^{-\alpha}(\mathring{h}_{rs}+\frac{H}{2}g_{rs})g^{ir}g^{ms}g^{jn}g^{kp}\mathring{h}_{ij,k}\mathring{h}_{mn,p} \nonumber\displaybreak[0]\\
 &~~~ -2H^{-\alpha}(\mathring{h}_{rs}+\frac{H}{2}g_{rs})g^{kr}g^{ps}g^{im}g^{jn}\mathring{h}_{ij,k}\mathring{h}_{mn,p} \nonumber\\
  \leq&2g(\partial_t\nabla \mathring{A},\nabla\mathring{A})+C|\nabla\mathring{A}|^2|\mathring{A}|-3H^{1-\alpha}|\nabla\mathring{A}|^2\nonumber\displaybreak[0]\\
   =&2g(\nabla\partial_t\mathring{A},\nabla\mathring{A})+\mathring{A}*\nabla(H^{-\alpha}A)*\nabla\mathring{A}+C|\nabla\mathring{A}|^2|\mathring{A}|-3H^{1-\alpha}|\nabla\mathring{A}|^2\nonumber\\
   \leq & 2\alpha H^{-\alpha-1}g(\nabla\Delta\mathring{A},\nabla\mathring{A})+\left(-3H^{1-\alpha}-\alpha H^{1-\alpha}+4\alpha H^{-\alpha-1}\right)|\nabla\mathring{A}|^2\nonumber\displaybreak[0]\\
   &+C\left(|\nabla^2\mathring{A}|+|\nabla^2H|\right)|\nabla\mathring{A}|^2+C|\nabla\mathring{A}|^4+C|\nabla\mathring{A}|^2\left(|\mathring{A}|^3+|\mathring{A}|^2+|\mathring{A}|\right)\nonumber\\
   \leq & \alpha H^{-\alpha-1}\Delta|\nabla\mathring{A}|^2-2\alpha H^{-\alpha-1}|\nabla^2\mathring{A}|^2\nonumber\displaybreak[0]\\
   &+\left(-3H^{1-\alpha}-\alpha H^{1-\alpha}+4\alpha H^{-\alpha-1}\right)|\nabla\mathring{A}|^2+C|\nabla\mathring{A}|^4\\
   &+C\left(|\nabla^2\mathring{A}|+|\nabla^2H|\right)|\nabla\mathring{A}|^2+C|\nabla\mathring{A}|^2\left(|H^2-4|+|\mathring{A}|^3+|\mathring{A}|^2+|\mathring{A}|\right),\nonumber
\end{align}
where the constant $C$ depends only on $\alpha,\epsilon_0$ and $\bar{r_0}-\underline{r}_0$. We claim that
\begin{equation}\label{na2-H0}
  |\nabla^2H|\leq C_{11}|\nabla^2\mathring{A}|
\end{equation}
for a universal constant $C_{11}$ depending on the dimension $n=2$. In fact, let $\{e_1,e_2\}$ be a local orthonormal frame of $M_t$. Then
\begin{align}\label{na2-H1}
  |\nabla^2H|^2= &\sum_{i,j=1}^2 (h_{11,ij}+h_{22,ij})^2 \leq  \sum_{i,j=1}^2\left(2h_{11,ij}^2+2h_{22,ij}^2\right).
\end{align}
Using the Codazzi equation $h_{kl,m}=h_{km,l}$ and the fact $\nabla^2h_{kl}=\nabla^2\mathring{h}_{kl}$ for any pair $k\neq l$, we have
\begin{align}
  h_{22,1j}&=h_{21,2j}=\mathring{h}_{21,2j},\quad \forall~j=1,2,\\
  h_{11,1j} =& \mathring{h}_{11,1j}+\frac 12H_{,1j} = \mathring{h}_{11,1j}+h_{22,1j}-\mathring{h}_{22,1j}\nonumber\\
  =& \mathring{h}_{11,1j}+\mathring{h}_{21,2j}-\mathring{h}_{22,1j},\quad \forall~j=1,2.
\end{align}
Similarly, we have
\begin{align}
  h_{22,2j}=& \mathring{h}_{22,2j}+\mathring{h}_{12,1j}-\mathring{h}_{11,2j},\quad \forall~j=1,2,\\
  h_{11,2j}=&\mathring{h}_{12,1j}\quad \forall~j=1,2.\label{na2-H3}
\end{align}
Then \eqref{na2-H0} follows by combining \eqref{na2-H1} -- \eqref{na2-H3}. In view of the estimates \eqref{s7-As-decay} and \eqref{s7-H-decay0}, \eqref{na2-H0} and applying the Cauchy-Schwartz inequality
$$
2\alpha H^{-\alpha-1}|\nabla^2\mathring{A}|^2+\frac{C^2\cdot C_{11}^2}{8\alpha} |\nabla\mathring{A}|^4H^{\alpha+1}\geq C\cdot C_{11}|\nabla^2\mathring{A}||\nabla\mathring{A}|^2
$$
to kill out the second order term $|\nabla^2\mathring{A}|^2$ in \eqref{s7-na-pf4}, we have
\begin{align}\label{s7-na-pf5}
  \partial_t|\nabla\mathring{A}|^2 \leq & \alpha H^{-\alpha-1}\Delta|\nabla\mathring{A}|^2
   +\left(-3\cdot 2^{1-\alpha}+C|M_0|^{-1}e^{-2^{(1-\alpha)}\cdot t}\right)|\nabla\mathring{A}|^2+C|\nabla\mathring{A}|^4,
\end{align}
where the constant $C$ depends only on $\alpha,\epsilon_0$, $\bar{r_0}-\underline{r}_0$ and $E$.

However, we can not obtain the estimate \eqref{lem-s7-nA} by applying the maximum principle to \eqref{s7-na-pf5}, as the coefficient of $|\nabla\mathring{A}|^4$ is a positive constant. To kill out this bad term, we will combine the following evolution equation of $|\mathring{A}|^2$, which follows from \eqref{s7-nA-pf1}, \eqref{flow-H^k-g} and the estimates \eqref{s7-As-decay}, \eqref{s7-H-decay0}.
\begin{align}
  \partial_t|\mathring{A}|^2  =&\alpha H^{-\alpha-1}\Delta|\mathring{A}|^2-2\alpha H^{-\alpha-1}|\nabla \mathring{A}|^2-2\alpha(\alpha+1)H^{-\alpha-2}\mathring{h}_{ij}\nabla^iH\nabla^jH\nonumber\\
   &+2\alpha H^{-\alpha-1}|\mathring{A}|^4-(\alpha+2)H^{1-\alpha}|\mathring{A}|^2+4\alpha H^{-\alpha-1}|\mathring{A}|^2\nonumber\\
   \leq &\alpha H^{-\alpha-1}\Delta|\mathring{A}|^2+(-2^{-\alpha}\cdot\alpha+C|M_0|^{-1}e^{-2^{(1-\alpha)}\cdot t})|\nabla\mathring{A}|^2\nonumber\\
   &+(-2^{2-\alpha}+C|M_0|^{-1}e^{-2^{(1-\alpha)}\cdot t})|\mathring{A}|^2,
\end{align}
where the constant $C$ depends only on $\alpha,\epsilon_0$, $\bar{r_0}-\underline{r}_0$ and $E$.
Define
\begin{equation}
  \psi_t=\log |\nabla\mathring{A}|^2+K|\mathring{A}|^2
\end{equation}
for a constant $K$ to be determined later. Then
\begin{align}
  \partial_t\psi_t \leq &\alpha H^{-\alpha-1}\Delta \psi_t+\alpha H^{-\alpha-1}|\nabla\log|\nabla\mathring{A}|^2|^2-3\cdot 2^{1-\alpha}+C|M_0|^{-1}e^{-2^{(1-\alpha)}\cdot t}\nonumber\\
  &+\left(C+K(-2^{-\alpha}\alpha+C|M_0|^{-1}e^{-2^{(1-\alpha)}\cdot t})\right)|\nabla\mathring{A}|^2\nonumber\\
  &+K(-2^{2-\alpha}+C|M_0|^{-1}e^{-2^{(1-\alpha)}\cdot t})|\mathring{A}|^2,
\end{align}
where the constant $C$ depends only on $\alpha,\epsilon_0$, $\bar{r_0}-\underline{r}_0$ and $E$.
At the critical point $p$ of $\psi_t$, we have
\begin{equation*}
    \nabla \log |\nabla\mathring{A}|^2=-K\nabla |\mathring{A}|^2,
\end{equation*}
hence at this $p$, we have
\begin{equation*}
\begin{aligned}
   | \nabla \log |\nabla\mathring{A}|^2|^2&=K^2|\nabla |\mathring{A}|^2|^2=K^2|2|\mathring{A}|\nabla |\mathring{A}||^2\\
   &=4K^2|\mathring{A}|^2|\nabla|\mathring{A}||^2\leq4K^2|\mathring{A}|^2|\nabla\mathring{A}|^2\\
   &\leq4K^2C_2|M_0|^{-2}e^{-2^{(2-\alpha)}\cdot t}|\nabla\mathring{A}|^2,
   \end{aligned}
\end{equation*}
where we used the inequality $|\nabla|\mathring{A}||^2\leq |\nabla\mathring{A}|^2$ in the first inequality and \eqref{s7-As-decay} in the last inequality.
Therefore, at $p$,
\begin{align}\label{s7-na-pf7}
  \partial_t\psi_t \leq &\alpha H^{-\alpha-1}\Delta \psi_t-3\cdot 2^{1-\alpha}+C|M_0|^{-1}e^{-2^{(1-\alpha)}\cdot t}\nonumber\\
  &+\left(C+K(-2^{-\alpha}\cdot\alpha+C|M_0|^{-1}e^{-2^{(1-\alpha)}\cdot t})+4K^2C_2|M_0|^{-2}e^{-2^{(2-\alpha)}\cdot t}\right)|\nabla\mathring{A}|^2\nonumber\\
  &+K C|M_0|^{-1}e^{-2^{(1-\alpha)}\cdot t} C_2 |M_0|^{-2}e^{-2^{(2-\alpha)}\cdot t}
\end{align}
for some  constant $C=C(\alpha,\epsilon_0,\bar{r_0}-\underline{r}_0,E)$.
We can choose $K={C}{\alpha}^{-1}2^{2+\alpha}$ and choose $s_0$ sufficiently large such that $|M_0|^{-1}<\min\{\frac{\alpha}{C\cdot2^{1+\alpha}},\frac{\alpha}{\sqrt{2C\cdot C_2}\cdot 2^{3+\alpha}}\}.$
Then we have
\begin{equation*}
C+K(-2^{-\alpha}\cdot\alpha+C|M_0|^{-1}e^{-2^{(1-\alpha)}\cdot t})+4K^2C_2|M_0|^{-2}e^{-2^{(2-\alpha)}\cdot t}~<~0
\end{equation*}
and
\begin{equation*}
C|M_0|^{-1}<\alpha \cdot2^{-1-\alpha},~K C|M_0|^{-1}C_2 |M_0|^{-2}<\alpha^2\cdot2^{-6-2\alpha}.
\end{equation*}
Therefore, at the critical point of $\psi_t$, we have
\begin{align}\label{s7-na-pf8}
  \partial_t\psi_t \leq &\alpha H^{-\alpha-1}\Delta \psi_t-3\cdot 2^{1-\alpha}+\alpha \cdot2^{-1-\alpha}e^{-2^{(1-\alpha)}\cdot t} +\alpha^2\cdot2^{-6-2\alpha}e^{-3\cdot2^{(1-\alpha)}\cdot t} .
\end{align}
Applying maximum principal to \eqref{s7-na-pf8}, we conclude that
\begin{equation}\label{s7-na-pf9}
  \psi_t\leq \max_{M_0}\psi_0-3\cdot 2^{1-\alpha}\cdot t+C,
\end{equation}
for some constant $C=C(\alpha,\epsilon_0,\bar{r_0}-\underline{r}_0, E)$. Then \eqref{lem-s7-nA} follows immediately from \eqref{s7-na-pf9} and the estimate \eqref{neves-est1}.
\endproof


\end{document}